\documentclass[12pt]{amsart}

\newtheorem{thm}{Theorem}[section]

\newtheorem{lem}[thm]{Lemma}

\theoremstyle{definition}

\theoremstyle{remark}

\usepackage{fancyhdr}
\usepackage{graphicx, amssymb}
\graphicspath{ {./images/} }
\usepackage{enumerate}
\usepackage{amsmath}
\usepackage[parfill]{parskip}
\usepackage[export]{adjustbox}
\usepackage{lipsum}
\newcommand\blfootnote[1]{%
  \begingroup
  \renewcommand\thefootnote{}\footnote{#1}%
  \addtocounter{footnote}{-1}%
  \endgroup
}

\numberwithin{figure}{section}
\usepackage[margin=1in]{geometry}

\begin{document}


\title{Unexpected Spectral Asymptotics for Wave Equations on certain Compact Spacetimes}


\author{Jonathan Fox}
\address{Department of Mathematics, Cornell University}
\email{jmf369@cornell.edu}


\author{Robert S. Strichartz}
\address{Department of Mathematics, Cornell University}
\email{str@math.cornell.edu}



\begin{abstract}
We study the spectral asymptotics of wave equations on certain compact spacetimes where some variant of the Weyl asymptotic law is valid. The simplest example is the spacetime $S^1 \times S^2$. For the Laplacian on $S^1 \times S^2$ the Weyl asymptotic law gives a growth rate $O(s^{3/2})$ for the eigenvalue counting function $N(s) = \#\{\lambda _j: 0 \leq \lambda _j \leq s\}$. For the wave operator there are two corresponding eigenvalue counting functions $N^{\pm}(s) = \#\{\lambda _j: 0 < \pm \lambda _j \leq s\}$ and they both have a growth rate of $O(s^2)$. More precisely there is a leading term $\frac{\pi^2}{4}s^2$ and a correction term of $as^{3/2}$ where the constant $a$ is different for $N^{\pm}$. These results are not robust, in that if we include a speed of propagation constant to the wave operator the result depends on number theoretic properties of the constant, and generalizations to $S^1 \times S^q$ are valid for $q$ even but not $q$ odd. We also examine some related examples.
\end{abstract}


 \maketitle 

\section{Introduction} 
The spectrum of the Laplacian on a compact Riemanninan manifold satisfies the well-known Weyl asymptotic law, and similar results hold for other elliptic operators. It is not expected that similar results hold for wave equations. And yet, sometimes the unexpected happens!\blfootnote{2010 AMS Mathematics subject Classification. Primary 35P20 35L05}\blfootnote{Key words and phrases: Spectral asymptotics, d'Alembertian wave equation, compact spacetime, eigenvalue counting function, Zoll surface, globally hypoelliptic}\blfootnote{Research of the second author supported by the National Science Foundation. Grant DMS-1162045.}

Perhaps the simplest example where this occurs is for the d'Alembertian wave operator $\Box = -\frac{\partial ^2}{\partial t ^2} + \bigtriangleup_x$ on the compact spacetime $S^1 \times S^2$. Here we know exactly what the eigenfunctions are, namely $e^{ijt}Y_k(x)$ for $j \in \mathbb{Z}$ and $Y_k$ a spherical harmonic of degree $k \geq 0$, with eigenvalue $\lambda = j^2 - k(k+1)$, and for each $k$ the multiplicity is $2k+1$. The key observation is that $k(k+1) = (k+\frac{1}{2})^2 - \frac{1}{4}$, so there is no possibility that $j^2$ and $k(k+1)$ can get close enough to completely cancel. Thus the eigenvalue $\lambda = 0$ occurs with multiplicity one, with $j=0$ and $k=0$. We form two eigenvalue counting functions

$N^+(s) = \# \{\lambda \in (0,s] \}$

$N^-(s) = \# \{\lambda \in [-s,0) \}$

for the positive and negative parts of the spectrum (counting multiplicty, of course), and observe that these are all finite. Note that this would not be the case if we considered the spacetime $S^1 \times S^3$, for then the spherical harmonics have eigenvalues $k(k+2) =$

$(k+1)^2 -1$, so the eigenvalue $\lambda = 1$ already has infinite multiplicity. We might say that there is "number theory" behind this dichotomy, as even dimensional spheres are like $S^2$ and odd dimensional spheres are like $S^3$. We will see more number theory at work when we consider d'Alembertians with a speed of propagation constant $\Box_c = -\frac{\partial ^2}{\partial t^2} + c^2 \bigtriangleup _x$ in section 4 below.

It is easy to see that the spectrum of $\Box$ consists exactly of all integers. Write $M^+(t)$ for the multiplicity of $\lambda = t$ and $M^-(t)$ for the multiplicity of $\lambda = -t$ where $t$ denotes any positive integer. Then the choice $j = \pm t$, $k = t-1$ gives the eigenvalue $\lambda = t$ with multiplicity $4t-2$, and the choice $j = \pm t$, $k = t$ gives the eigenvalue $\lambda = -t$ with multiplicity $4t + 2$. So we have lower bounds

$M^+(t) \geq 4t-2$

$M^-(t) \geq 4t+2$

In Figures 1.1 and 1.2 we show the graphs of $M^+(t)$ and $M^-(t)$ and in Figures 1.3 and 1.4 we show the graphs of $\frac{M^+(t) - (4t-2)}{t}$ and $\frac{M^-(t) - (4t+2)}{t}$. 
\begin{figure}
	\hspace*{-.35in}
    \includegraphics[scale = .8]{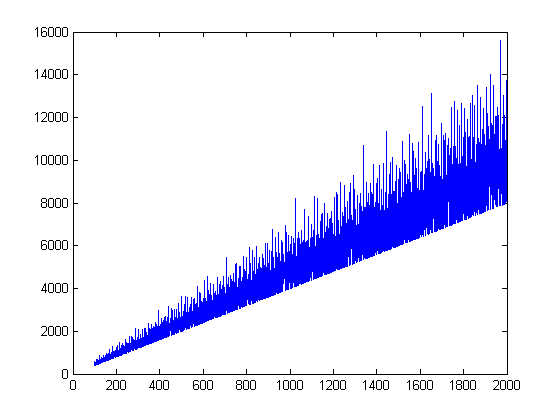}
   \caption{$M^+(t)$ on $[100,2000]$}
\end{figure}
\begin{figure}
	\hspace*{-.35in}
    \includegraphics[scale = .8]{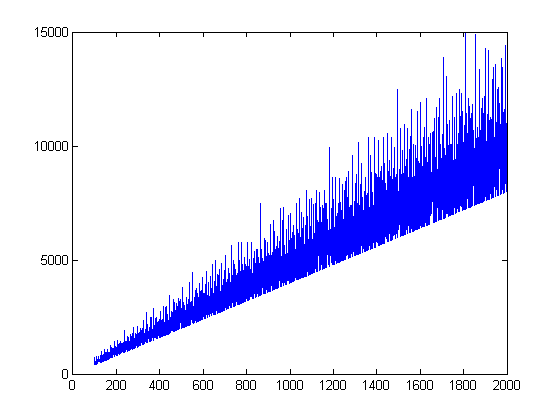}
   \caption{$M^-(t)$ on $[100,2000]$}
\end{figure}
\begin{figure}
	\hspace*{-.35in}
    \includegraphics[scale = .8]{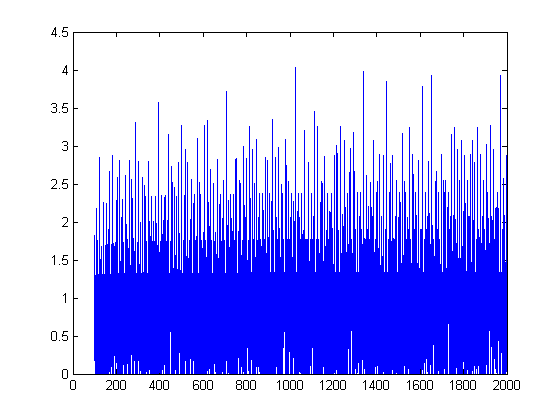}
   \caption{$(M^+(t)-(4t-2))/t$ on $[100,2000]$}
\end{figure}
\begin{figure}
	\hspace*{-.35in}
    \includegraphics[scale = .8]{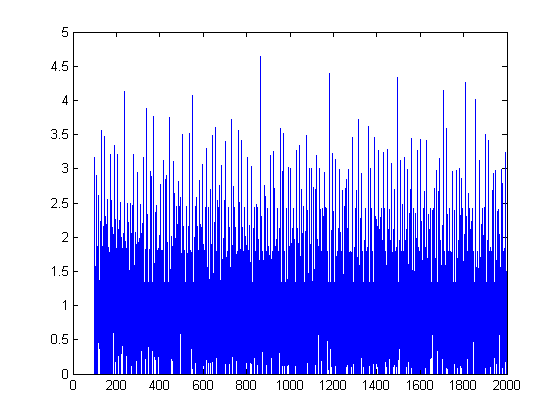}
   \caption{$(M^-(t)-(4t+2))/t$ on $[100,2000]$}
\end{figure}
We might speculate that these are perhaps bounded functions, but it might be difficult to settle this question. Nevertheless it is quite clear that they are close enough to being bounded that $N^{\pm}(s)$ should have growth rate on the order of $s^2$. Note that this is different from the growth rate of $s^{3/2}$ of the eigenvalue counting function of the Laplacian $\frac{\partial ^2}{\partial 	t^2} + \bigtriangleup _x$ given by Weyl's law. Of course there is a simple heuristic explanation that the growth rate is faster because the eigenvalues are smaller, but it is not obvious how to explain the power $2$ without doing the computation. The surprising result, established in section 2, is the asymptotic expression $N^{\pm}(s) \propto \frac{\pi^2}{4}s^2$. In fact there is even a correction term of order $s^{3/2}$, with a different constant for $N^+$ and $N^-$. We have no idea how to "explain" the constant $\frac{\pi^2}{4}$. (Confession: we discovered the constant numerically to many decimal places, then "looked it up" on the internet, and after receiving the verdict "$\frac{\pi^2}{4}$" we found the proof.) Needless to say, $\frac{\pi^2}{4} > 2$, consistent with our lower bounds for $M^{\pm}$.

There is a somewhat related property, called \textit{global hypoellipticity} ([GW]) shared by many of our operators. Recall that if $L$ is any operator that commutes with the Laplacian $\bigtriangleup$ on a compact manifold and so shares a complete set of eigenfunctions, the lower bound 

$(1.1)$ $|\lambda_L| \geq c \lambda_{\bigtriangleup}^{\alpha /2}$ for all $\lambda_{\bigtriangleup} > 0$, for fixed $c$ and $\alpha$, (here $\lambda_L$ and $\lambda_{\bigtriangleup}$ denote the eigenvalues for $L$ and $\bigtriangleup$ associated with any common eigenfunction) implies that $L$ has a parametrix (or resolvant) that is smoothing of order $\alpha$ in the scale of $L^2 -$ Sobolev spaces. In particular, if $Lu$ is $C^{\infty}$ then $u$ is $C^{\infty}$, but more precisely, if $Lu \in H^{s}$ then $u \in H^{s+\alpha}$ where $H^{s}$ denotes the $L^2$-Sobolev space. Such operators are not necessarily hypoellptic in the usual local sense. Wave equations never are. Also they typically do not exhibit smoothing for $L^{p}$-Sobolev spaces. Our basic example satisfies $(1.1)$ with $\alpha = 1$, as this is simply the lower bound

$(1.2)$ $|j^2 - k(k+1)| \geq c(j^2 + k(k+1))^{1/2}$.

Note that when $j=k$ this is the obvious bound $k \geq c(2k^2 + k)^{1/2}$. We leave it to the reader to verify $(1.2)$ by considering separately the cases $|j| > 2k$, $|j| < \frac{1}{2}k$, and $\frac{1}{2}k \leq |j| \leq 2k$. 

It might be tempting to try to relate the exponent $\alpha$ in $(1.1)$ with the power in the asymptotics of $N^{\pm}(s)$. But there is no such relation, since most of our examples satisfy $(1.1)$	with $\alpha = 1$ but have different powers in the asymptotics of $N^{\pm}(s)$. Even worse, there are examples (the first example in section 3 and $L_2$ in section 6) that are not globally hypoelliptic because the 0-eigenspace is infinite dimensional, yet nevertheless have power law asymptotics for $N^{\pm}(s)$. 

In this paper we give the spectral asymptotics for the following examples:
\begin{itemize}
  \item $\Box$ on $S^1 \times S^2$, in Section 2.
  \item $\Box$ on $S^1 \times S^1$, and $\Box + i\frac{\partial}{\partial x}$, in section 3. The first has an infinite dimensional 0-eigenspace, but $N^{\pm}(s)$ have asymptotics of the form $s\log s + (2\gamma - 1)s + O(s^{1/2})$ where $\gamma$ denotes Euler's constant. The second has a 1-dimensional 0-eigenspace and $N^{\pm}(s)$ have asymptotics $s\log s + (4\log 2 + 2\gamma - 1)s + O(s^{1/2})$.
  \item $\Box_c = -\frac{\partial ^2}{\partial t^2} + c^2 \bigtriangleup _x$ on $S^1 \times S^2$, in section 4. We require $c$ to be a rational number with odd numerator, and get asymptotics similar to $\Box$ but multiplied by $c^{-2}$. For rational numbers with even numerator and generic irrational numbers it is easy to see that no such asymptotic is possible. 
  \item The ultrahyperbolic operator $\Box = -\bigtriangleup _x + \bigtriangleup _y$ on products of spheres $S^{p} \times S^{q}$, where $p$ is odd and $q$ is even, in section 5. Aside from the case $p = 1$, $q=2$ discussed in section 2, we find asymptotics for $N^{\pm}(s) = c(p,q)s^{p+q-1} + O(s^{p+q-2})$ for certain specific constants $c(p,q)$ given in terms of the zeta function $\zeta(p+q-1)$. Note that $p + q -1$ is an even integer so the zeta values are known. Note that these higher dimensional examples are better behaved than $S^1 \times S^2$ as there is no $s^{p+q-3/2}$ term in the asymptotics.
  \item The fourth order operators $L_1 = (\frac{\partial}{\partial t})^4 + \bigtriangleup _x$ and $L_2 = -\frac{\partial ^2}{\partial t ^2} - \bigtriangleup ^2 _x$ on $S^1 \times S^2$, in section 6. Note that these operators are not hypoelliptic. For $L_1$ the 0-eigenspace has multiplicity one, and $N^{\pm}(s)$ are $O(s^{3/2})$. For $L_2$ the 0-eigenspace is infinite dimensional but we have more precise $s \log s$ asymptotics for $N^{\pm}(s)$.
\end{itemize}

There are other related examples of wave operators with similar spectral asymptotics. On $S^1 \times S^q$ for $q$ even we may consider the k-forms wave operator. In other words we consider a function $u(t)$ that takes values in the k-forms on $S^q$, and the eigenfunction equation is $(-\frac{\partial ^2}{\partial t^2} + \bigtriangleup^{(k)})u(t) = \lambda u(t)$, where $\bigtriangleup^{(k)}$ is the de Rham Laplacian on k-forms. The spectrum of $\bigtriangleup^{(k)}$ on $S^q$ is described explicitly in \cite{F}. The eigenvalues all have form $m^2 + (q-1)m + k(q-k-1)$ or $m^2 + (q-1)m + (k-1)(q-k)$ where $m$ varies over the nonnegative integers. In other words, they are just translates by fixed constants of the eigenvalues of the function Laplacian. Thus the methods of section 5 may be applied. The multiplicities of the eigenspaces are not given explicitly in [F], but can be deduced from the representation theory of the group $SO(q+1)$. We will not attempt to compute the exact asymptotics here, except to note that in the case $q=2$ the spectrum for 2-forms is identical to the spectrum for functions, and the spectrum for 1-forms is the union of the two (except for the 0-eigenspace).

The last example arises if we replace $S^2$ with the standard metric by $S^2$ with a Zoll surface metric (see \cite{G1}) with the corresponding Laplacian $\bigtriangleup _Z$. The Zoll surfaces have the property that all geodesics are closed and have length $2\pi$. It was observed in \cite{W} and \cite{UZ} that the spectrum of $\bigtriangleup _Z$ is just a bounded perturbation of the spectrum of $\bigtriangleup$, so the $(2k+1)$-dimensional eigenspace with eigenvalue $k(k+1)$ is replaced with a cluster of $2k+1$ eigenvalues in the interval $[k(k+1) - M, k(k+1) + M]$ (here M is a constant that depends on the Zoll surface metric). This leads to bounds $N^{\pm}(s-M) \leq N_Z^{\pm}(s) \leq N^{\pm}(s+M)$ so we can transfer asymptotics for $N^{\pm}$ to asymptotics for $N_Z^{\pm}$.
The key observation that lies behind our work is that the spectrum of the Laplacian on a sphere has gaps. It is natural to ask why this is so. The sphere has a large nonabelian group of symmetries, and this implies that eigenspaces have large multiplicities. However, the existence of a large group of symmetries does not imply gaps in the spectrum. For a counterexample just take the product of a sphere with a generic compact manifold. There is still a large symmetry group and high multiplicities arising from the sphere factor, but the generic factor fills in all potential gaps. The Zoll surfaces suggest that gaps arise when all geodesics are closed and have the same length. This is essentially proven in \cite{H} (see also \cite{G2}).

Our results are usually expressed by writing $N^{\pm}(s)$ as the sum of a specific function of $s$ plus a remainder $R_{N^{\pm}}(s)$ together with an asymptotic estimate for the remainder. Often the remainder takes on both positive and negative values, so that we can hope that averaging will result in large cancellations and hence better asymptotic estimates. We define the average remainder by:

$(1.3)$ $AR_{N^{\pm}}(t) = \frac{1}{t}\int\limits_0^t R_{N^{\pm}}(s)\,ds$.

This idea works very well for the Laplacians, as seen in [1] and [3]. For example, the Laplacians on the 2-torus $S^1 \times S^1$, the eigenvalue counting function $N(s)$ is equal to the number of lattice points inside a disk of radius $\sqrt{s}$. The remainder $R_N(s) = N(s) - \pi s$ is conjectured to be $O(s^{1/4})$, but this is a major unsolved problem. In [1] it is shown that $AR_{N}(t) = O(t^{-1/4})$, and even more precise statements are shown in terms of a specific almost periodic function.

For each of our examples we compute the averages $AR_{N^{\pm}}(t)$. This usually clarifies the behavior of $R_{N^{\pm}}$, and in some cases it allows us to add an additional lower order term to our approximation. These observations are just experimental. The key technical idea in the proofs in [8] is the use of the Poisson summation formula. While it is not inconceivable that a similar method could be used in our examples, it is not straightforward to carry this out. In section 2, the averaging reduces the amount of oscillation, and for $N^-$ it provides strong evidence of a periodic oscillation. This reinforces the conjecture that the remainder is actually $O(s)$. For the first example in Section 3, the average helped us estimate the growth rate of the remainder to be $s^{.0919}$ rather than $s^{.5}$. Without the averaging the oscillations make the growth rate difficult to estimate. For the second example in section 3, the growth rate of $O(\sqrt{s})$ seems correct, and the average makes this quite apparent, with the possibility of a limit. The data in section 4 parallels that from section 2. In section 5 the advantage of averaging is very striking. It allows us to guess a lower order term in the asymptotics that is completely invisible without averaging. Finally, in section 6, averaging appears to produce limits that would not exist without it.

In the theory of Laplacians on fractals, there are many examples with even larger spectral gaps than those occurring on spheres (see \cite{S1}). This led to the observation in \cite{BS} that on the product of two copies of the Sierpinski gasket, there are operators of the form $\bigtriangleup ' - c\bigtriangleup ''$ (here $\bigtriangleup '$ and $\bigtriangleup ''$ denote the Laplacians on each of the factors), for the appropriate choice of the positive constant $c$, that have the same spectral asymptotics as $\bigtriangleup '$ + $\bigtriangleup ''$. It was later shown in \cite{IRS} that these operators are elliptic pseudodifferential operators in the appropriate sense. These examples, while different in important ways, are similar in flavor to our results, and were one of the inspirations for this work. Another inspiration came from \cite{KK}, which describes a different type of unexpected spectral property of wave equations on compact quotients of anti-de Sitter spacetime.

In this paper we show graphically some of the results of our numerical computations. The website \cite{JF} has much more data, and the programs that were used to generate the data.
\section{The wave equation on $S^{1} \times S^{2}$}
Our simplest example is the wave operator $\Box = -\frac{\partial^{2}}{\partial t ^{2}} + \bigtriangleup _{x}$ for $t \in S^{1}$ and $x \in S^{2}$. The eigenfunctions are just $e^{ijt}Y_{k}(x)$ where $Y_{k}$ denotes a spherical harmonic of degree $k$. The spectrum consists of the values:

$(2.1)$ $j^{2} - k(k+1)$ for $j \in \mathbb{Z}$ and $k \geq 0$ with multiplicity $2k + 1$. It is easy to see that the 0-eigenspace just consists of the constants, so has multiplicity one. The two eigenvalue counting functions are:

$(2.2)$  $N^{+}(s) = \sum (2k+1)$ on $0 < j^{2} - k(k+1) \leq s$

$(2.3)$  $N^{-}(s) = \sum (2k+1)$ on $0 < k(k+1) - j^{2} \leq s$ 

Because the eigenvalues are integers, $N^{\pm}(s) = N^{\pm}(\lfloor s \rfloor)$. We will usually assume that $s$ is an integer.

\begin{lem}
$ $\newline
$(2.4)$:
\begin{align*}
N^{+}(s) = 2\sum\limits_{n=1}^{\lfloor\sqrt{s}\rfloor} \lfloor \frac{s - (n-1)^{2}}{2n - 1} \rfloor ^2
\end{align*}
\end{lem}
\begin{proof}
Note that the condition $0 < j^{2} - k(k+1)$ is equivalent to $k + 1 \leq |j|$, while the condition $j^{2} - k(k+1) \leq s$ is equivalent to $|j| \leq \lfloor \sqrt{s + k(k+1)} \rfloor$.

Define the integer n by:

$(2.5)$ $\lfloor \sqrt{s + k(k+1)} \rfloor = n + k$

Since $j \neq 0$ we may group together the terms corresponding to $\pm j$ for $j > 0$ to obtain:

$(2.6)$ $N^{+}(s) = \sum 2n(2k+1)$ 

for the appropriate values of $k$ and $n$. In fact, $(2.5)$ defines $n$ in terms of $k$, but we wish to fix $n$ and determine the values of $k$ that yield the given value of $n$. We write $(2.5)$ as $\sqrt{s + k(k+1)} = n + k + \delta$ for $0 \leq \delta < 1$ which simplifies to $k = \frac{s - (n + \delta)^{2}}{2n - 1 + 2\delta}$.

Since this is a decreasing function of $\delta$ and $k$ is an integer we obtain the range:

$(2.7)$ $ \lfloor \frac{s - (n+1)^{2}}{2n+1} + 1 \rfloor \leq k \leq \lfloor\frac{s - n^{2}}{2n - 1} \rfloor$

Using the identity $\sum\limits_{a}^{b} (2k+1) = (b+1)^{2} - a^{2}$ we find that the sum over $k$ in $(2.6)$ is:

\begin{flalign*}
(2.8) \hspace{.05 in} \lfloor \frac{s-(n-1)^{2}}{2n-1} \rfloor ^{2} - \lfloor\frac{s-n^{2}}{2n+1}\rfloor^{2} = a_{n}^{2} - a_{n+1}^{2}
\end{flalign*}

for $a_{n} = \lfloor \frac{s-(n-1)^{2}}{2n-1} \rfloor$.

The upper bound $n \leq \sqrt{s}$ comes from $(2.7)$. Then $(2.4)$ follows from
 
$\sum\limits_{n=1}^{\sqrt{s}}2n(a_{n}^{2} - a_{n+1}^{2}) = \sum\limits_{n=1}^{\sqrt{s}}2a_{n}^{2}$.
\end{proof}

\begin{thm}
We have the asymptotic formula

$(2.9)$ $N^{+}(s) = \frac{\pi ^{2}}{4}s^{2} - \frac{4}{3}s^{3/2} + R_{N^{+}}(s)$ with the remainder estimate

$(2.10)$ $R_{N^{+}}(s) = O(s \log(s) )$ as $s \rightarrow \infty$.
\end{thm}

\begin{proof}
Define $\eta_{n}$ in $[0,1)$ by 

$(2.11)$ $\lfloor \frac{s-(n-1)^{2}}{2n-1} \rfloor = \frac{s-(n-1)^{2}}{2n-1} - \eta_{n}$. Then $(2.4)$ becomes 

$(2.12)$ $N^{+}(s) = \sum\limits_{n=1}^{\sqrt{s}} \frac{2s^{2}-4s(n-1)^{2} + 2(n-1)^{4}}{(2n-1)^{2}} - \sum\limits_{n=1}^{\sqrt{s}} \frac{4\eta_{n}(s-(n-1)^{2})}{2n-1} + \sum\limits_{n=1}^{\sqrt{s}}2\eta^{2}_{n}$.

The last two sums in $(2.12)$ are exactly seen to be $O(s\log(s))$. We write the first easily as

$(2.13)$ $s^{2}\sum\limits_{n=1}^{\infty} \frac{2}{(2n-1)^{2}} - s^{2}\sum\limits_{n=\lfloor \sqrt{s} \rfloor + 1}^{\infty} \frac{2}{(2n-1)^{2}} - s\sum\limits_{n=1}^{\lfloor\sqrt{s}\rfloor} \frac{4(n-1)^{2}}{(2n-1)^{2}} + \sum\limits_{n=1}^{\lfloor\sqrt{s} \rfloor} \frac{2(n-1)^{4}}{(2n-1)^{2}}$ 

The first term in $(2.13)$ is exactly $\frac{\pi^{2}}{4}s^{2}$, while the other terms are $-\frac{1}{2}s^{3/2} + O(s^{1/2})$, $-s^{3/2} + O(s\log(s))$ and $\frac{1}{6}s^{3/2} + O(s)$. Adding them up yields $(2.9)$ and $(2.10)$.
\end{proof}

In Figure 2.1 we show the graph of $R_{N^+(s)}/s$. 
\begin{figure}
	\hspace*{-.35in}
    \includegraphics[scale = .8]{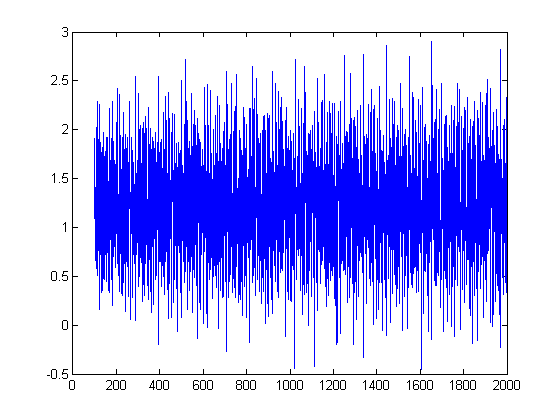}
   \caption{$R_{N^{+}}(s)/s$ from $(2.9)$ on $[100,2000]$}
\end{figure} 
This suggests that the error estimate should be $R_{N^+(s)} = O(s)$ rather than $(2.10)$. We can give a heuristic argument for this as follows. The $O(s\log{s})$ terms in $R_{N^+(s)}$ come from $-s\sum\limits_{n=1}^{\sqrt{s}}\frac{4\eta_n}{2n-1}$ in $(2.12)$ and $s\sum\limits_{n=1}^{\sqrt{s}}(1 - \frac{4(n-1)^2}{(2n-1)^2}) = s\sum\limits_{n=1}^{\sqrt{s}}\frac{4n-3}{(2n-1)^2} = s\log{s} + O(s^{1/2})$ in $(2.13)$. It is reasonable to expect that $\eta_n$ averages to $1/2$, so $-s\sum\limits_{n=1}^{\sqrt{s}}\frac{4\eta_n}{2n-1}$ should behave like $-s\sum\limits_{n=1}^{\sqrt{s}}\frac{2}{2n-1} = -s\log{s} + O(s)$, so the $s\log{s}$ terms would cancel.

In Figure 2.2 we show the graph of $(AR_{N^+}(s) - cs)/s$ for $c = .6154$ determined experimentally. This suggests a more refined asymptotic formula:

\begin{figure}
	\hspace*{-.35in}
    \includegraphics[scale = .8]{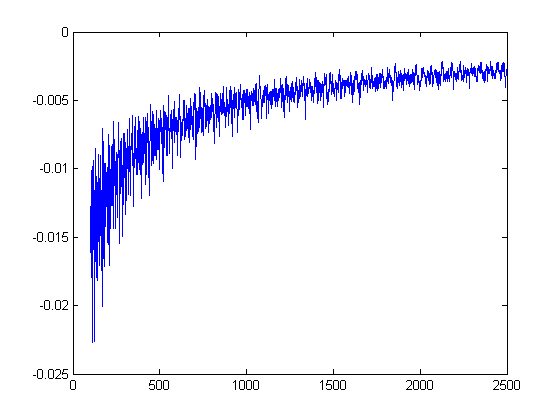}
   \caption{$(AR_{N^+}(s) - cs)/s$ for $c = .6154$ on $[100,2500]$}
\end{figure} 

$(2.9') N^+(s) = \frac{\pi^2}{4}s^2 - \frac{4}{3}s^{3/2} + 2cs + \tilde{R}_{N^+}(s)$ where $A\tilde{R}_{N^+}(s) = o(s)$, but we are not able to give decisive experimental evidence for this asymptotic estimate.

\begin{lem}
$(2.14)$ $N^-(s) = \sum\limits_{k=0}^{\lfloor \sqrt{s + 1/4} - 1/2 \rfloor}(2k+1)^2 + \sum\limits_{n=1}^{\lfloor \sqrt{s + 1/4} + 1/2 \rfloor}2\lfloor \frac{s+n^2}{2n-1} \rfloor^2$.
\end{lem}

\begin{proof}
We note that the condition $j^2 < k(k+1)$ is equivalent to $|j| < k$. To analyze the condition $j^2 \geq k(k+1) - s$ we consider two cases: 
\begin{enumerate}[I.]
\item
If $s \geq k(k+1)$ the condition is always satisfied, so there are exactly $(2k+1)$ values of $j$ satisfying $|j| \leq k$ and so the total contribution to $N^-(s)$ from this case is the first sum on the right side of $(2.14)$.
\item
If $s < k(k+1)$ then the condition is $|j| \geq \sqrt{k(k+1) - s}$. Since $|j|$ must be an integer we have $\lceil \sqrt{k(k+1) - s} \rceil \leq |j| \leq k$ so the total number of such $j$ is $2(k + 1 - \lceil \sqrt{k(k+1) - s} \rceil) = 2[k+1 - \sqrt{k(k+1) -s}]$. The condition $\sqrt{k(k+1) - s} \leq k$ means $k \leq s$, while $s < k(k+1)$ means $\sqrt{s + \frac{1}{4}} - \frac{1}{2} < k$ or $[\sqrt{s + \frac{1}{4}} + \frac{1}{2}] \leq k$. Thus the total contribution of these terms to $N^-(s)$ is 

$(2.15)$ $\sum\limits_{k = \lfloor \sqrt{s + 1/4} + 1/2 \rfloor}^{s} 2\lfloor k + 1 - \sqrt{k(k+1) - s} \rfloor (2k+1)$. To complete the proof we need to show that $(2.15)$ is equal to the second sum on the right side of $(2.14)$. To do this we define the integer $n = \lfloor k + 1 - \sqrt{k(k+1) -s} \rfloor$, and ask which values of $k$ correspond to a fixed value of $n$. This means $n = k + 1 - \sqrt{k(k+1) - s} - \delta$ for $0 \leq \delta < 1$, hence 

\begin{align*}
(2.16) \hspace{.05 in} k = \frac{s + (n - 1 + \delta)^2}{2n-1 + 2\delta} 
\end{align*}

From $(2.16)$ we obtain $\lfloor \frac{s + n^2}{2n+1} \rfloor + 1 \leq k \leq \lfloor \frac{s + (n-1)^2}{2n-1} \rfloor$ and this yields the second sum in $(2.15)$ as in the proof of Lemma 2.1.
\end{enumerate}
\end{proof}

\begin{thm}
We have the asymptotic formula

$(2.17)$ $N^-(s) = \frac{\pi^2}{4}s^2 + 2s^{3/2} + R_{N^-(s)}$ with the remainder estimate

$(2.18)$ $R_{N^-(s)} = O(s\log{s})$ as $s \rightarrow \infty$.
\end{thm}

\begin{proof}
The proof is similar to the proof of Theorem 2.2. The different coefficient of $s^{3/2}$ arises as $2 = \frac{4}{3} - \frac{1}{2} + 1 + \frac{1}{6}$, with $\frac{4}{3}$ coming from the first sum in $(2.14)$, and the other terms arising as in the proof of Theorem 2.2 except for the change in sign in the cross-terms of $(s + n^2)^2$ as opposed to $(s - (n-1)^2)^2$.
\end{proof}

In Figure 2.3 we show the graph of $R_{N^-(s)}/s$, supporting the conjecture that it is bounded. In Figure 2.4 we show the graph of $AR_{N^-}(s)/s$. This suggests that a more refined asymptotic formula would be:

$(2.17')$ $N^-(s) = \frac{\pi^2}{4}s^2 + 2s^{3/2} + g(s)s + \tilde{R}_{N^-(s)}$ 

where $g$ is a periodic (or almost periodic) function, but it is not clear what estimate the new remainder $\tilde{R}_{N^-}$ should satisfy. We have no explanation for why the $N^-$ counting function exhibits this interesting structure as compared with $N^+$. 
\begin{figure}
	\hspace*{-.35in}
    \includegraphics[scale = .8]{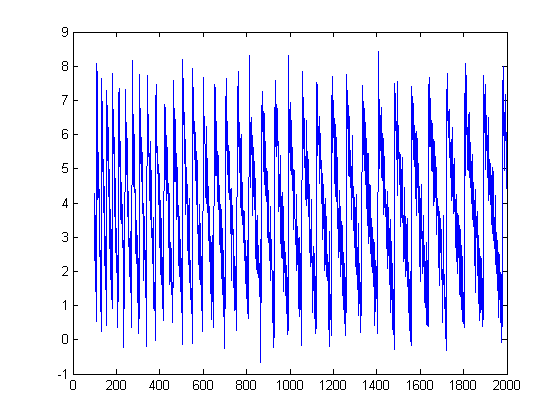}
   \caption{$R_{N^-(s)}/s$ from $(2.17)$ on $[100,2000]$}
\end{figure}
\begin{figure}
	\hspace*{-.35in}
    \includegraphics[scale = .8]{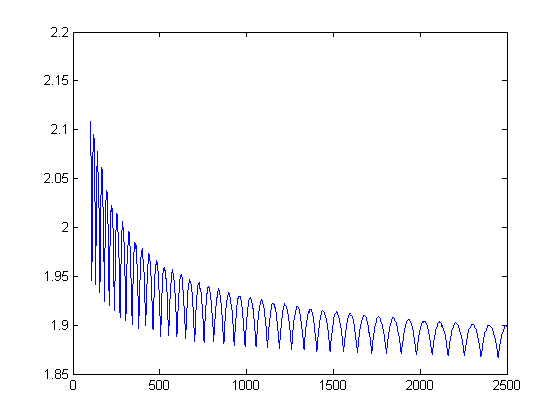}
   \caption{$AR_{N^-}(s)/s$ on $[100,2500]$}
\end{figure}
\section{Wave equation on $S^{1} \times S^{1}$}
The wave operator $\Box = -\frac{\partial ^{2}}{\partial t} + \frac{\partial ^2}{\partial x^2}$ for $t \in S^{1}$ and $x \in S^{1}$ has eigenfunctions $e^{ijt}e^{ikx}$ with eigenvalues $j^{2} - k^{2}$ for $j,k \in \mathbb{Z}$. Obviously it has an infinite dimensional 0-eigenspace corresponding to $j = k$, so it is different in this respect from $\Box$ on $S^{1} \times S^{2}$. Aside from this, we can  study the behavior of $N^{\pm}(s)$ as before. Here $N^{+}(s) = N^{-}(s)$ is obvious by interchanging $j$ and $k$. The value of $N^{+}(s)$ is then just the number of solutions of the inequalities $0 < j^{2} - k^{2} \leq s$. Note that $k^{2} < j^{2}$ is equivalent to $|k| + 1 \leq |j|$ and $j^{2} \leq s + k^{2}$ is equivalent to $|j| \leq \lfloor \sqrt{s + k^{2}} \rfloor$. So:

$(3.1)$ $N^{+}(s) = 2\lfloor \sqrt{s} \rfloor + 4 \sum\limits_{n=1}^{\lfloor \frac{s-1}{2} \rfloor} \left( \lfloor \sqrt{s+k^{2}} \rfloor - k \right)$, where the first term corresponds to $k=0$ and the sum groups together $\pm j$, $\pm k$. The upper bound for $k$ comes from the requirement that $k+1 \leq \sqrt{s + k^{2}}$.

\begin{lem}
$(3.2)$ $N^{+}(s) = 2\lfloor \sqrt{s} \rfloor + 4 \sum\limits_{n=1}^{\lfloor \sqrt{s} \rfloor} \lfloor \frac{s-n^{2}}{2n} \rfloor$
\end{lem}

\begin{proof}
Define $n$ by $\lfloor \sqrt{s+k^{2}} = n+k$, or 

$(3.3)$ $\sqrt{s+k^{2}} = n + k + \delta$ for $0 \leq \delta < 1$. If we solve $(3.3)$ we obtain $k = \frac{s-(n+\delta)^{2}}{2(n+\delta)}$, so for fixed $n$ we require $\frac{s-(n+1)^{2}}{2(n+1)} < k \leq \frac{s-n^{2}}{2n}$, so $(3.1)$ becomes

$N^+(s) = 2\lfloor \sqrt{s} \rfloor + 4\sum\limits_{n=1}^{\lfloor \sqrt{s+1} -1 \rfloor}n\left( \lfloor \frac{s - n^2}{2n} \rfloor - \lfloor \frac{s - (n+1)^2}{2(n+1)} \rfloor \right) = 2\lfloor \sqrt{s} \rfloor + 4 \sum\limits_{n=1}^{\lfloor \sqrt{s+1}-1 \rfloor}\lfloor\frac{s-n^2}{2n}\rfloor$.

This is equivalent to $(3.2)$.
\end{proof}

\begin{thm}
We have the asymptotic expansion

$(3.4)$ $N^{+}(s) = s\log s + (2\gamma - 1)s + R_{N^{+}}(s)$ as $s \rightarrow \infty$ where 

$(3.5)$ $R_{N^{+}}(s) = O(\sqrt{s})$. Here $\gamma$ denotes Euler's constant. 
\end{thm}

\begin{proof}
Write $\frac{s-n^{2}}{2n} = \lfloor \frac{s-n^{2}}{2n} + \eta_{n}$ for $0 \leq \eta_{n} < 1$. Then

$N^{+}(s) = 4\sum\limits_{n=1}^{\lfloor\sqrt{s}\rfloor}\left(\frac{s}{2n} - \frac{n}{2}\right) - 4\sum\limits_{n=1}^{\lfloor\sqrt{s}\rfloor} \eta_{n} + 2\lfloor \sqrt{s} \rfloor$

Note that $2s\sum\limits_{n=1}^{\lfloor \sqrt{s} \rfloor}\frac{1}{n} = 2s(\log{\lfloor \sqrt{s} \rfloor} + \gamma + O(\frac{1}{\sqrt{s}}))$ and $2\sum\limits_{n=1}^{\lfloor \sqrt{s} \rfloor}n = \lfloor \sqrt{s} \rfloor^2 + \lfloor \sqrt{s} \rfloor$. Adding everything up yields $(3.4)$ with the remainder estimate $(3.5)$. \end{proof}
Figure 3.1 shows that the graph of $R_{N^+(s)}/s^{\alpha}$, where $\alpha = .0919$ was determined experimentally. 
\begin{figure}
	\hspace*{-.35in}
    \includegraphics[scale = .8]{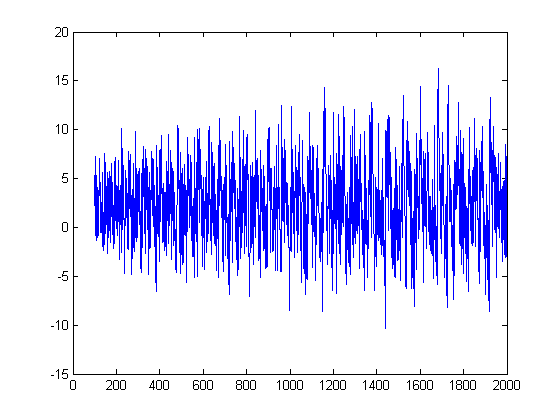}
   \caption{$R_{N^+(s)}/s^{.0919}$ from $(3.4)$ on $[100,2000]$}
\end{figure}
\begin{figure}
	\hspace*{-.35in}
    \includegraphics[scale = .8]{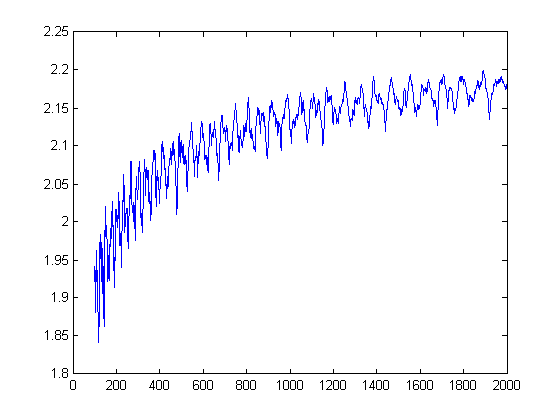}
   \caption{$AR_{N^+}(s)/s^{.0919}$ from $(3.4)$ on $[100,2000]$}
\end{figure}
In Figure 3.2 we show the graph of $AR_{N^+}(s)/s^{\alpha}$. 
This gives experimental evidence that error estimate $(3.5)$ can be greatly improved.

We can overcome the problem with the infinite dimensional 0-eigenspace by considering the modified wave operator:

$(3.6)$ $\Box' = \Box + i \frac{\partial}{\partial x}$

The eigenfunctions are the same, but the eigenvalues are now $j^{2} - k(k+1)$. In other words, we have the same eigenvalues as for $\Box$ on $S^{1} \times S^{2}$, but the multiplicity is one rather than $(2k+1)$, so the 0-eigenspace consists of the constants, and has dimension one. The definition of $N^{+}(s)$ leads to 

$(3.1')$ $N^{+}(s) = 2\lfloor \sqrt{s} \rfloor + 4\sum\limits_{n=1}^{\lfloor \sqrt{s} \rfloor} \left( \lfloor \sqrt{s + k(k+1)} - k \right)$ in place of $(3.1)$. Reasoning as in Lemma 3.1 yields 

$(3.2')$ $N^{+}(s) = 2\lfloor \sqrt{s} \rfloor + 4 \sum\limits_{n=1}^{\lfloor\sqrt{s}\rfloor} \lfloor \frac{s-n^{2}}{2n-1} \rfloor$ in place of $(3.2)$. We note that the difference between $(3.2')$ and $(3.2)$ is $4\sum\limits_{n=1}^{\lfloor \sqrt{s} \rfloor} \frac{s}{2n(2n-1)} = (4\log2)s + O(\sqrt{s})$ since $\sum\limits_{n=1}^{\infty} \frac{1}{2n(2n-1)} = \log 2$. So instead of $(3.4)$ we have 

$(3.4')$ $N^{+}(s) = s \log s + (4\log 2 + 2 \gamma - 1)s + R_{N^{+}}(s)$ with the same error estimate $(3.5)$. 

In this case $N^{-}(s)$ is not equal to $N^{+}(s)$, and we will compute it by interchanging the roles of $j$ and $k$. Note that we may restrict the values of $k$ to $k \geq 0$ and double, since $k$ and $-k-1$ generate the same eigenvalue. For $j = 0$ we have the condition $0 < k(k+1) \leq s$ so this  contributes $2(\lfloor \sqrt{s} \rfloor -1)$ to $N^{-}(s)$. For $j > 0$ we have the conditions $j \leq k$ and $k(k+1) \leq s + j^{2}$, which is equivalent to $k \leq \lfloor \sqrt{s + \frac{1}{4} + j^{2}} - \frac{1}{2} \rfloor$. This yields

$(3.7)$ $N^{-}(s) = 2(\lfloor \sqrt{s}\rfloor - 1) + 4\sum\limits_{n=1}^{s} \left( \lfloor \sqrt{s + \frac{1}{4} + j^{2}} + \frac{1}{2} \rfloor - j \right)$ in place of $(3.1)$. If we set $n = \lfloor \sqrt{s+\frac{1}{4} + j^{2}} + \frac{1}{2}\rfloor - j$, so $\sqrt{s + \frac{1}{4}+j^{2}} = n + j + \eta_{j}$, then $\frac{s-n^{2}-n}{2n+1} < j < \frac{s-n^{2}+n}{2n-1}$, so 
\begin{align*}
(3.8) \hspace{.05 in} \space \space \space \space N^{-}(s) = 2(\lfloor \sqrt{s} \rfloor - 1) + 4\sum\limits_{n=1}^{\lfloor\sqrt{s}\rfloor}n\left(\lfloor \frac{s-n^{2} + n}{2n-1}\rfloor - \lfloor \frac{s-n^{2}-n}{2n+1}\rfloor \right) 
\\
= 2(\lfloor \sqrt{s} \rfloor - 1) + 4\sum\limits_{n=1}^{\lfloor\sqrt{s} \rfloor} \lfloor \frac{s-n^{2}+n}{2n-1} \rfloor
\end{align*}
in place of $(3.2)$. Comparing this with $(3.2')$ we see the difference is $O(\sqrt{s})$, so $N^{-}(s)$ satisfies the same asymptotics $(3.4')$ as $N^{+}(s)$.

In Figure 3.3 we show the graph of $R_{N^+}(s)/\sqrt{s}$ from $(3.4')$. This suggests that the errorestimate cannot be improved. Figure 3.4 is the graph of its average divided by $\sqrt{s}$ suggesting that perhaps the limit as $s \rightarrow \infty$ exists.
\begin{figure}
	\hspace*{-.35in}
    \includegraphics[scale = .8]{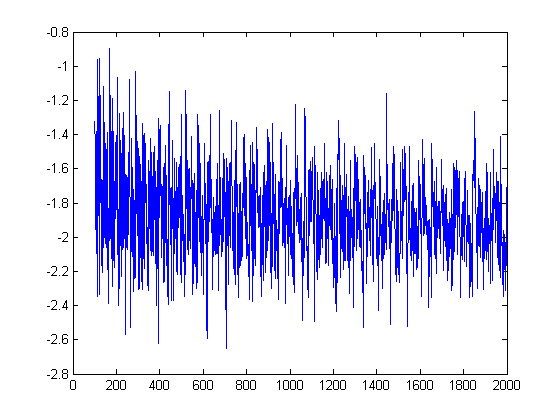}
   \caption{$R_{N^+(s)}/\sqrt{s}$ from $(3.4')$ on $[100,2000]$}
\end{figure}
\begin{figure}
	\hspace*{-.35in}
    \includegraphics[scale = .8]{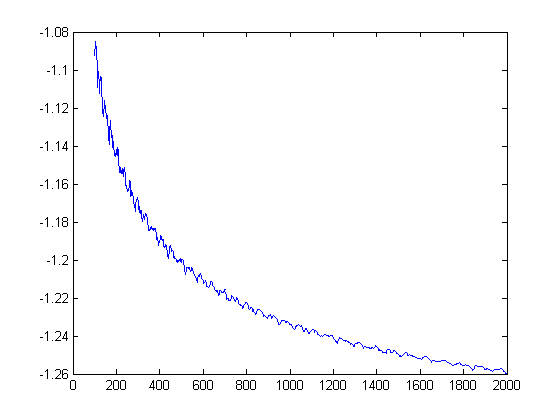}
   \caption{$AR_{N^+}(s)/\sqrt{s}$ from $(3.4')$ on $[100,2000]$}
\end{figure}
\section{Wave equation with velocity constant}
In this section we discuss the wave equation $\Box_{c} = -\frac{\partial^{2}}{\partial t ^{2}} + c^{2}\bigtriangleup _{x}$ on $S^{1} \times S^{2}$ and how its spectral asymptotics depends on the velocity constant $c$. The eigenfunctions and multiplicities are the same as the case $c = 1$ discussed in section 2, but the eigenvalues are $j^2 - c^2k(k+1) = j^2 - (ck + \frac{c}{2})^{2} + \frac{c^2}{4}$.

We consider the first case when $c$ is rational. We see immediately that if $c = \frac{2p}{2q+1}$ then there are infinitely many solutions of $(2q+1)j = (2k+1)p$ and so the eigenspace with eigenvalue $\frac{c^2}{4}$ has infinite multiplicity. Thus we restrict attention to the case $c = \frac{2p+1}{q}$ where $q$ may be even or odd but is relatively prime to $2p+1$. We define the eigenvalue counting functions $N_{c}^{+}(s)$ and $N_{c}^{-}(s)$ as before, with 

$(4.1)$ $N_c^+(s) = \sum (2k+1)$ on $0 < j^2 - c^2k(k+1) \leq s$ and 

$(4.2)$ $N_c^-(s) = \sum (2k+1)$ on $0 < c^2k(k+1) - j^2 \leq s$ in place of $(2.2)$ and $(2.3)$. Again, under our assumptions on c, the 0-eigenspace consists of just the constants. For simplicity we consider the first case when $q =1$, so $c = 2p + 1$.

\begin{lem}
$(4.3)$ $N_{2p+1}^{+}(s) = 2\sum\limits_{n=1}^{\lceil \sqrt{s} \rceil + p} \lfloor \frac{s-(n-p-1)^{2}}{(2p+1)(2n-1)}\rfloor ^2 + O(1)$
\end{lem}

\begin{proof}
As in the proof of Lemma 2.1 we observe that $0 < j^2 - (2p+1)^2k(k+1)$ is equivalent to $(2p+1)k+p+1 \leq |j|$ for $k \geq p^2$ and $(2p+1)k+p \leq |j|$ for $k < p^2$. Also $j^2 - (2p+1)^2k(k+1) \leq s$ is equivalent to $|j| \leq \lfloor \sqrt{s + (2p+1)^2k(k+1)}\rfloor$. So we define $n$ by 

$(4.4)$ $\lfloor \sqrt{s+ (2p+1)^2k(k+1)} \rfloor = (2p+1)k + p + n$ if $k \geq p^2$ or $(2p+1)k + p - 1 + n$ if $k < p^2$ and obtain 

$(4.5)$ $N_{2p+1}^{+}(s) = \sum 2n(2k+1)$ for the appropriate values of $n$ and $k$. Note that we may use the first formula in $(4.4)$ for all $k$ at the cost of an error of $O(1)$. For fixed $n$ we write $(4.4)$ as $\sqrt{s + (2p+1)^2k(k+1)} = (2p+1)k + p + n + \delta$ for $0 \leq \delta < 1$, which simplifies to $k = \frac{s-(n+p + \delta)^2}{(2p+1)(2n+2\delta - 1)}$. Thus the range of $k$ is 

$(4.6)$ $\lfloor \frac{s-(n+p+1)^2}{(2p+1)(2n + + 1)} + 1 \rfloor \leq k \leq \lfloor \frac{s-(n+p)^2}{(2p+1)(2n-1)} \rfloor$, so the sum over $k$ in $(4.5)$ yields $a_{n}^2 - a_{n+1}^2$ for $a_n = \lfloor \frac{s-(n-p-1)^2}{(2p+1)(2n-1)} \rfloor$. The rest of the proof is exactly as in Lemma 2.1.
\end{proof}

\begin{thm}
We have the asymptotic formula 

$(4.7)$ $N_{2p+1}^{+}(s) = \frac{1}{(2p+1)^2}(\frac{\pi^2}{4}s^2 - \frac{4}{3}s^{3/2})+ R_{N_{2p+1}^{+}}(s)$ with the remainder estimate 

$(4.8)$ $R_{N_{2p+1}^{+}}(s) = O(s \log s)$ as $s \rightarrow \infty$.
\end{thm}

\begin{proof}
The proof is almost identical to the proof of Theorem 2.2. The factor of $(2p+1)^2$ in the denominator of $(4.3)$ leads to the same factor in $(4.7)$. The appearance of $p$ in the numerator and in the upper sum limit only contributes to the remainder term.
\end{proof}
Figure 4.1 shows the graph of $R_{N_{2p+1}^{+}}(s)/s$ for $p=5$, and Figure 4.2 shows the graph of the average, also divided by $s$. These suggest that the error is $O(s)$, and that the limit as $s \rightarrow \infty$ of $AR_{N^+_{2p+1}}(s)/s$ exists.
\begin{figure}
	\hspace*{-.35in}
    \includegraphics[scale = .8]{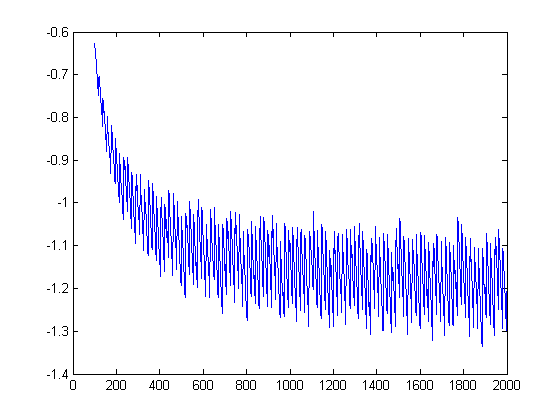}
   \caption{$R_{N_{2p+1}^{+}}(s)/s$ from $(4.7)$ on $[100,2000]$ with $p=5$}
\end{figure}
\begin{figure}
	\hspace*{-.35in}
    \includegraphics[scale = .8]{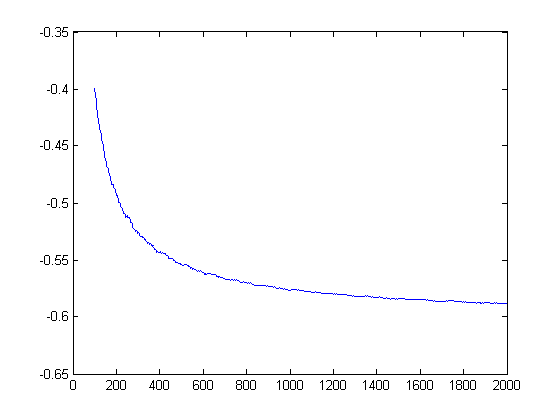}
   \caption{$AR_{N_{2p+1}^{+}}(s)/s$ from $(4.7)$ on $[100,2000]$ with $p = 5$}
\end{figure}
\begin{lem}
$(4.9)$ $N_{2p+1}^{-}(s) = \sum\limits_{k = 0}^{\lfloor \sqrt{\frac{s}{(2p+1)^2} + \frac{1}{4}} - \frac{1}{2}}((2p+1)2k + 1)(2k+1) + \sum\limits_{n=1}^{\lfloor \sqrt{s + \frac{1}{4}(2p+1)^2} + \frac{1}{2} \rfloor} 2\lfloor \frac{s + n^2}{(2p+1)(2n-1)} \rfloor ^2$
\end{lem}

\begin{proof}
As in the proof of Lemma 2.3, we note that $j^2 \leq (2p+1)^2k(k+1)$ is equivalent to $|j| \leq (2p+1)k$. To analyze the condition $j^2 \geq (2p+1)^2k(k+1) - s$ we consider two cases: 
\begin{enumerate}[I.]
\item
If $s \geq (2p+1)^2k(k+1)$ the condition is always satisfied, so there are exactly $(2p+1)2k + 1$ values of $j$ satisfying $|j| \leq (2p+1)k$, so this case contributes the first sum to $(4.9)$.
\item
If $s < (2p+1)^2k(k+1)$, then the conditions are $\lceil \sqrt{(2p+1)^2k(k+1) - s} \rceil \leq |j| \leq (2p+1)k$, so the total number of such $j$ is $2\lfloor (2p+1)k + 1 - \sqrt{(2p+1)^2k(k+1) - s} \rfloor$. The condition $\sqrt{(2p+1)^2k(k+1) - s} \leq (2p+1)k$ means $k \leq \frac{s}{(2p+1)^2}$, while $s < (2p+1)^2k(k+1)$ means $[ \sqrt{\frac{s}{(2p+1)^2} + \frac{1}{4}} + \frac{1}{2}] \leq k$, so the total contribution to $N_{2p+1}^{-}(s)$ of these terms is 

$(4.10)$ $\sum\limits_{k = [\sqrt{\frac{s}{(2p+1)^2} + \frac{1}{4}} + \frac{1}{2}]}^{\frac{s}{(2p+1)^2}} 2\lfloor (2p+1)k+1 - \sqrt{(2p+1)^2k(k+1)-s}\rfloor(2k+1)$
\end{enumerate}
Now we define $n = \lfloor (2p+1)k+1 - \sqrt{(2p+1)^2k(k+1) - s}\rfloor$ so that 

$n = (2p+1)k+1 - \sqrt{(2p+1)^2k(k+1) - s} - \delta$ for $0 \leq \delta < 1$, and solve to obtain

$(4.11)$ $k = \frac{s+(n-1+\delta)^2}{(2p+1)(2n-1+2\delta)}$ which differs from $(2.16)$ only by the factor $(2p+1)$ in the denominator. So the range of $k$ is $\lfloor \frac{s + n^2}{(2p+1)(2n+1)}\rfloor + 1 \leq k \leq \lfloor \frac{s + (n-1)^2}{(2p+1)(2n+1)}\rfloor$. Thus $(4.10)$ is equal to the second sum in $(4.9)$. 
\end{proof}

\begin{thm}
We have the asymptotic formula 

$(4.12)$ $N_{2p+1}^-(s) = \frac{1}{(2p+1)^2}\frac{\pi^2}{4}s^2 + \frac{2}{(2p+1)^2}s^{3/2} + R_{N_{2p+1}^{-}}(s)$ with the remainder estimate 

$(4.13)$ $R_{N_{2p+1}^{-}}(s) = O(s\log s)$ as $s \rightarrow \infty$.
\end{thm}

\begin{proof}
The proof is the same as Theorem 2.4 except for the factor $(2p+1)^2$ in the denominator.
\end{proof}
In Figure 4.3 we show the graph of $R_{N_{2p+1}^{-}}(s)/s$ and in Figure 4.4 the average for $p=1$. These are similar to Figures 2.3 and 2.4.
\begin{figure}
	\hspace*{-.35in}
    \includegraphics[scale = .8]{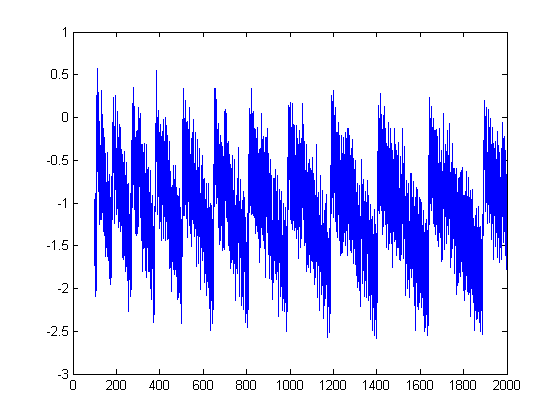}
   \caption{$R_{N_{2p+1}^{-}}(s)/s$ from $(4.12)$ on $[100,2000]$ with $p = 1$}
\end{figure}
\begin{figure}
	\hspace*{-.35in}
    \includegraphics[scale = .8]{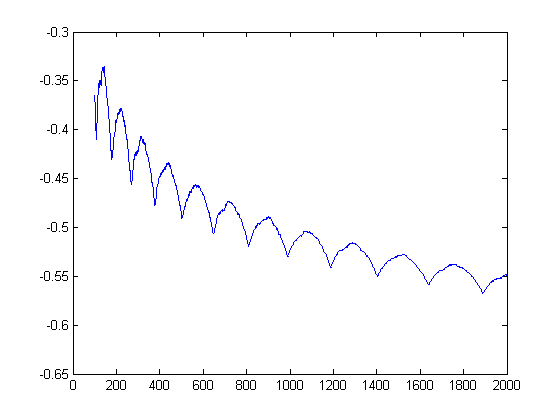}
   \caption{$AR_{N_{2p+1}^{-}}(s)/s$ from $(4.12)$ on $[100,2000]$ with $p = 1$}
\end{figure}

Similar reasoning shows that for $c = \frac{2p+1}{q}$ we have 

$(4.14)$: 

$N_c^+(s) = \frac{1}{c^2}(\frac{\pi^2}{4}s^2 - \frac{4}{3}s^{3/2}) + R_{N_c^+}(s)$

$N_c^-(s) = \frac{1}{c^2}(\frac{\pi^2}{4}s^2 + 2s^{3/2}) + R_{N_c^-(s)}$, both with the same error estimate.

Finally we consider the case of irrational values of $c$. Suppose $\left\{b_0, b_1, b_2...\right\}$ are the coefficients of the continued fraction expansion of $c$, and $\frac{p_i}{q_i}$ the associated rational approximations of $c$. In particular, we know 

$(4.15)$ $|\frac{p_i}{q_i} - c| < \frac{1}{q_iq_{i+1}}$. 

From the previous discussion it is clear that the essential issue is whether or not the numerators $p_i$ are even. Indeed if $p_i = 2p_i'$ is even, so $q_i = 2q_i' + 1$ is odd, then the choice $j = p'$, $k = q'$ leads to $j^2 - c^2k(k+1) = j^2 - (\frac{p_i}{q_i})^2k(k+1) + O(1)$ by $(4.15)$ and $j^2 - (\frac{p_i}{q_i})^2k(k+1) = (\frac{p_i'}{2q_i' +1})^2$ so there is a uniform bound for the eigenvalue. Thus if there are infinitely many even $p_i'$s, then there are infinitely many eigenvalues in a bounded region. It is easy to see that this is the generic case, because if $p_{i-2}$ and $p_{i-1}$ are odd, then the choice of $b_i$ odd leads to $p_i = b_ip_{i-1} + p_{i-2}$ even.

There remains the exceptional case when all but  a finite number of the $b_i'$s are even. In half of these cases, all but a finite number of the $p_i$ will be odd (for example, if $b_0$ is even and all $b_i$ for $i \geq 1$ are even). It is plausible that there are some such examples where the estimate $(4.14)$ holds, especially if the coefficients $b_i$ grow rapidly so the approximations $(4.15)$ are very close. However, to actually prove this would be rather delicate since it would require careful control of the error terms $(4.8)$ and $(4.13)$. We will not attempt to do this here.
\section{Wave equation on $S^p \times S^q$}
In this section we discuss the wave equation $\Box = -\bigtriangleup_x + \bigtriangleup_y$ on products of spheres $x \in S^p$ and $y \in S^q$ where $p$ is odd and $q$ is even. The eigenfunctions are products of spherical harmonics $Y_j^p(x)Y_k^q(y)$ with eigenvalues $j(j+p-1) - k(k+q-1) = (j+\frac{p-1}{2})^2 - (k+\frac{q-1}{2})^2 - (\frac{p-1}{2})^2 + (\frac{q-1}{2})^2$. The parity assumptions easily imply that the 0-eigenspace is finite dimensional and $N^+(s)$ and $N^-(s)$ are always finite. The dimension of the space of spherical harmonics $Y_k^q$ is 

$(5.1)$ $	\binom{q+k}{k} - \binom{q - 2 +k}{k-2} = \frac{2k+q-1}{q-1} \binom{k+q-2}{k} = \frac{2}{(q-1)!}k^{q-1} + \frac{q-1}{(q-2)!}k^{q-2} + O(k^{q-3})$, and a similar formula for $Y_j^p$ except when $p=1$. So, when $p = 1$ we have

$(5.2)$ $N^+(s)= \sum(\frac{2k+q-1}{q-1})\binom{k+q-2}{k}$ over all $j \in \mathbb{Z}$ and $k \geq 0$ satisfying 

$0 < j^2 - k(k+q-1) \leq s$, 

$(5.3)$ $N^-(s) = \sum (\frac{2k+q-1}{q-1})\binom{k+q-2}{k}$ over all $j \in \mathbb{Z}$ and $k \geq 0$ satisfying 

$0 < k(k+q-1)-j^2 \leq s$. When $p \geq 3$ we have 

$(5.4)$ $N^+(s) = \sum(\frac{2j+p-1}{p-1}\binom{j+p-2}{j}(\frac{2k+q-1}{q-1})\binom{k+q-2}{k}$ over all $j \geq 0$, $k \geq 0$ satisfying $0 < j(j+p-1) - k(k+q-1) \leq s$, 

$(5.5)$ $N^-(s) = \sum(\frac{2j+p-1}{p-1})\binom{j+p-2}{j}(\frac{2k+q-1}{q-1})\binom{k+q-2}{k}$ over all $j \geq 0$, $k \geq 0$ satisfying $0 < k(k+q-1) - j(j+p-1) \leq s$.

\begin{lem}
Suppose $p = 1$ and $q$ is even. Then

$(5.6)$ $N^+(s) = \frac{4}{q!}\sum\limits_{n=1}^{\lfloor \sqrt{s} \rfloor - \frac{q}{2} + 1}\lfloor \frac{s-(n+q/2 -1)^2}{2n-1} \rfloor ^q + \frac{2q}{(q-1)!}\sum\limits_{n=1}^{\lfloor \sqrt{s} \rfloor - \frac{q}{2} + 1} \lfloor \frac{s - (n+q/2-1)^2}{2n-1} \rfloor^{q-1} + O(s^{q-2})$
\end{lem}

\begin{proof}
The condition $j^2 > k(k+q-1)$ is equivalent to $|j| \geq k + \frac{q}{2}$, while the condition $j^2 \leq k(k+q-1)+s$ is equivalent to $|j| \leq \lfloor \sqrt{k(k+q-1) + s} \rfloor$. Thus for each $k$ the total number of $j$ is $2(\lfloor \sqrt{k(k+q-1) + s} \rfloor - k - q/2 + 1)$, and for this to be nonzero we must have $k \leq s - (\frac{q}{2})^2$. So $(5.2)$ becomes

$(5.7)$ $N^+(s) = \sum\limits_{k=0}^{s-(\frac{q}{2})^2} 2(\lfloor \sqrt{k(k+q-1) + s} - k - \frac{q}{2} + 1)(\frac{2k+q-1}{q-1})\binom{k+q-2}{k}$. 

Now we can define $n$ by $\sqrt{k(k+q-1)+s} = n + k + q/2 - 1 + \delta$ with $0 \leq \delta < 1$.

We solve this for $k$ to obtain $k = \frac{s-(n-1 + q/2 + \delta)^2}{2n-1 + 2\delta}$ so for fixed $n$ we have $k$ in the range 

$(5.8)$ $\lfloor \frac{s-(n+\frac{q}{2})^2}{2n+1} \rfloor \leq k \leq \lfloor \frac{s - (n+q/2 - 1)^2}{2n-1} \rfloor$.

To find the sum over $k$ we define the polynomial $Q_q(m)$ of degree $q$ by the equation

$(5.9)$ $\sum\limits_{k=0}^{m}(\frac{2k+q-1}{q-1})\binom{k+q-2}{k} = Q_q(m)$. Then $(5.7)$ becomes

$(5.10)$ $N^+(s) = \sum\limits_{n=1}^{\lfloor \sqrt{s} \rfloor - \frac{q}{2} +1}2nQ_q(\lfloor \frac{s-(n+q/2 - 1)^2}{2n-1} \rfloor) - Q_q( \lfloor \frac{s - (n + q/2)^2}{2n+1} \rfloor) = \sum\limits_{n=1}^{\lfloor \sqrt{s} \rfloor - \frac{q}{2} +1}2Q_q(\lfloor \frac{s - (n + q/2 - 1)^2}{2n-1} \rfloor )$.

To identify $(5.6)$ with $(5.10)$ we just have to compute the two leading terms of $Q_q$ from $(5.1)$, namely $\sum\limits_{k=0}^{m}(\frac{2}{(q-1)!}k^{q-1} + \frac{q-1}{(q-2)!}k^{q-2}) = \frac{2}{q!}m^q + \frac{1}{(q-1)!}m^{q-1} + \frac{1}{(q-2)!}m^{q-1} + O(m^{q-2})$.
\end{proof}
\begin{thm}
Suppose $p=1$ and $q\geq 4$ is even. Then 

$(5.11)$ $N^+(s) = \frac{4}{q!}(1-2^{-q})\zeta(q)s^q + O(s^{q-1})$.
\end{thm}
\begin{proof}
Write $\lfloor \frac{s-(n + \frac{q}{2} - 1)^2}{2n-1} \rfloor = \frac{s-(n+\frac{q}{2} - 1)^2}{2n-1} - \eta_n$ with $0 \leq \eta_n < 1$. Then

$(5.12)$ $\lfloor \frac{s-(n+\frac{q}{2}-1)^2}{2n-1} \rfloor^q = \frac{s^q}{(2n-1)^q} + \frac{qs^{q-1}}{(2n-1)^q}((n + \frac{q}{2} - 1)^2 + \eta_n(2n-1)) + O(s^{q-2})$.

Since $q \geq 4$ we obtain from $(5.6)$ and $(5.12)$ that $N^+(s) = \frac{4}{q!} \sum\limits_{n=1}^{\lfloor \sqrt{s} \rfloor - \frac{q}{2} +1} \frac{s^q}{(2n-1)^q} + O(s^{q-1}) = \frac{4}{q!}((1 - 2^{-q})\zeta(q)s^q - \sum\limits_{n = \lfloor \sqrt{s} \rfloor - \frac{q}{2} + 2}^{\infty} \frac{s^q}{(2n-1)^q}) + O(s^{q-1})$ and the last sum is $O(s^{q-\frac{q-1}{2}})$, which is $O(s^{q-1})$ since $q \geq 4$. 
\end{proof}
Figure 5.1 shows the graph of $R_{N^+}(s)/s$ for $p=1$, $q=4$, suggesting strongly that the limit as $s \rightarrow \infty$ does not exist. However, our data suggests that for the average $AR_{N^+}(s)$ the limit of $AR_{N^+}(s)/s^3$ does exist and is approximately equal to $-.347275$. Moreover, $AR_{N^+}(s) + .347275s^3$ appears to have a growth rate of $s^{\beta}$ for $\beta$ approximately equal to $2.6976$. Figure 5.2 shows the graph of $(AR_{N^+}(s) + .347275s^3)/s^{2.6976}$. This suggests that it might be possible to include a term of order $s^{q-1}$ in $(5.11)$ so

$(5.11')$ $N^+(s) = \frac{4}{q!}(1-2^{-q})\zeta(q)s^q + c(q)s^{q-1} + \tilde{R}_{N^+}(s)$ and $A\tilde{R}_{N^+}(s) = O(s^{q-1 - \delta(q))}$ for some choice of $c(q)$ and $\delta(q) > 0$. (So for $q=4$ we have approximately $c(4) \approx 4(.347275)$ and $\delta(q) \approx .3024$).) Note that $\tilde{R}_{N^+}(s)$ would only be $O(s^{q-1})$, and only by averaging would the $s^{q-1}$ term make a difference.   
\begin{figure}
	\hspace*{-.35in}
    \includegraphics[scale = .8]{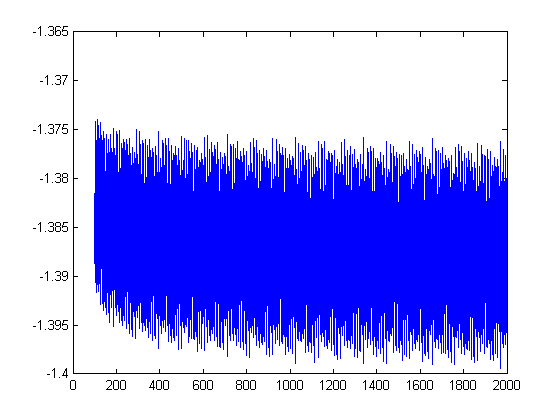}
   \caption{$R_{N^{+}}(s)/s^3$ from $(5.11)$ on $[100,2000]$ with $p = 1$ and $q = 4$}
\end{figure}
\begin{figure}
	\hspace*{-.35in}
    \includegraphics[scale = .8]{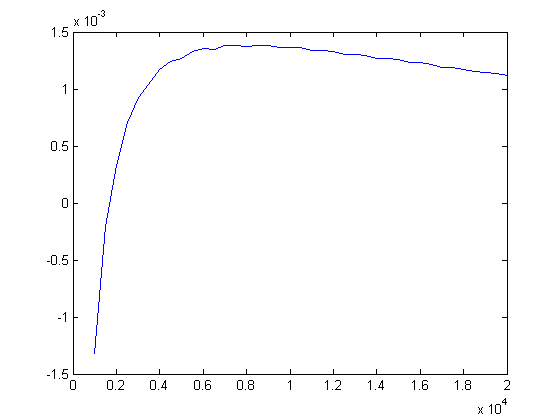}
   \caption{$(AR_{N^{+}}(s) + .347275*s^3)/s^{2.6976}$ from $(5.11)$ on $[1000,20000]$ with $p = 1$ and $q = 4$}
\end{figure}
\begin{thm}
Suppose $p=1$ and $q \geq 4$ is even. Then $N^-(s)$ satisfies the same estimate $(5.11)$ as $N^+(s)$.
\end{thm}

\begin{proof}
We sketch the proof, leaving out the details about many of the terms that only contribute to the error term. We note that the condition $j^2 < k(k+q-1)$ is essentially equivalent to $|j| \leq k + \frac{q}{2} - 1$ (we need $k \geq (\frac{q}{2}-1)^2 + 1$ for this to be exact). The condition $j^2 \geq k(k+q-1) - s$ leads to two cases:
\begin{enumerate}[I.]
\item
If $s \geq k(k+q-1)$ then this condition is always valid, and this leads to the number of $j$ values equal to $2k + q - 1$, and $k$ is bounded by essentially $\sqrt{s}$, so the total contribution is on the order of $\sum\limits_{k=1}^{\sqrt{s}}k^q = O(s^{(q+1)/2})$ which is $O(s^{q-1})$ since $q \geq 4$.
\item
If $s < k(k+q-1)$, then the condition is equivalent to $|j| \geq \lceil \sqrt{k(k+q-1) - s} \rceil$, so the number of $j$ values is $2\lfloor k + \frac{q}{2} - \sqrt{k(k+q-1)} - s\rfloor$, and this is nonzero when $k \leq \frac{1}{2}(s + (\frac{q}{2}-1)^2)$. We define $n$ by the equation $\sqrt{k(k+q-1)-s} = k + \frac{q}{2} - \delta - n$ with $0 \leq \delta < 1$, and solve for $k$ to obtain $k = \frac{s + (n + \frac{q}{2}+\delta)^2}{2n-1+2\delta}$. For fixed $n$ the range of $k$ is $\lfloor \frac{s + (n+1+ \frac{q}{2})^2}{2n+1} \leq k \leq \lfloor \frac{s + (n+\frac{q}{2})^2}{2n-1} \rfloor$, and the largest value of $n$ is on the order of $\sqrt{s}$. The rest of the proof is the same as for $N^+(s)$.
\end{enumerate}
\end{proof}
\begin{thm}
Suppose $p\geq 3$ is odd and $q$ is even. Then 

$(5.13)$ $N^{\pm}(s) = \frac{4}{(p+q-1)(p-1)!(q-1)!}(1-2^{-(p+q-1)})\zeta(p+q-1)s^{p+q-1} + O(s^{p+q-2})$.
\end{thm}
\begin{proof}
By $(5.4)$ and $(5.1)$, the leading term of $N^+(s)$ is 

$(5.14)$ $\frac{4}{(p-1)!(q-1)!}\sum k^{q-1}j^{p-1}$ over $j \geq 0$ and $k \geq 0$ satisfying $0 < j(j+p-1) - k(k+q-1) \leq s$. We will see that the lower order terms contribute only to the remainder $O(s^{p+q-2})$. The condition $k(k+q-1) < j(j+p-1)$ is equivalent to $j \geq k + \frac{q-p+1}{2}$ for $k$ sufficiently large. The condition $j(j+p-1) \leq s + k(k+q-1)$ is equivalent to $j \leq \lfloor \sqrt{s + k(k+q-1)} \rfloor - \frac{p-1}{2}$ (with a finite number of exceptions). If we fix $k$ and sum over $j$ then $(5.14)$ is equal to 

$(5.15)$ $\frac{4}{p!(q-1)!}\sum\limits_{k=0}^{s - (\frac{q}{2})^2}k^{q-1} ((\lfloor \sqrt{s + k(k+q-1)} \rfloor - \frac{p-1}{2})^p - (k + \frac{q-p-1}{2})^p)$ plus lower order terms. So we define $n$ by 

$(5.16)$ $\sqrt{s + k(k+q-1)} = n + k + \frac{q-p-1}{2} + \delta$ with $0 \leq \delta < 1$, and $(5.15)$ becomes

$(5.17)$ $\frac{4}{p!(q-1)!} \sum\limits_{k=0}^{s-(\frac{q}{2})^2}k^{q-1}((n + k + \frac{q-p-1}{2})^p - (k + \frac{q - p -1}{2})^p)$. Note that when $k = s - (\frac{q}{2})^2$ then $n = \frac{p+1}{2}$, and when $k =0$ then $n = \lfloor \sqrt{s} \rfloor - \frac{q-p-1}{2}$, so that gives us the range of $n$. For fixed $n$ in this range we solve $(5.16)$ for $k$ to obtain $k = \frac{s - (n + \frac{q-p+1}{2} + \delta)^2}{2n - p + 2\delta}$ so the range of $k$ is 

$(5.18)$ $1 + \lfloor \frac{s - (n+1 + \frac{q - p + 1}{2})^2}{2n + 2 -p} \rfloor \leq k \leq \lfloor \frac{s - (n + \frac{q - p + 1}{2})^2}{2n-p} \rfloor$. We expand $(n + k + \frac{q-p-1}{2})^p - (k + \frac{q - p - 1}{2})^p = \sum\limits_{m=0}^{p-1} \binom{p}{m}n^{p-m}(k + \frac{q-p-1}{2})^m = pnk^{p-1} + \sum c_{kl}k^mn^l$ where the sum includes values of $m \leq p-2$ and $m + l \leq p$. As we will see, the sum only contributes to the remainder term. Taking the sum over $k$ in $(5.17)$ using the principal term $pnk^{p-1}$ produces 

$\frac{4}{(p+q-1)(p-1)!(q-1)!}\sum\limits_{n= \frac{p+1}{2}}^{\lfloor \sqrt{s} \rfloor -  \frac{q-p-1}{2}}n(a_n^{p+q-1} - a_{n-1}^{p+q-1})$ plus lower order terms, for $a_n = \lfloor \frac{s - (n + \frac{q-p+1}{2})^2}{2n-p} \rfloor$. Thus we obtain

$(5.19)$ $N^+(s) = \frac{4}{(p+q-1)(p-1)!(q-1)!}\sum\limits_{n = \frac{p+1}{2}}^{\lfloor \sqrt{s} \rfloor - \frac{q-p-1}{2}} \lfloor \frac{s- (n+\frac{q-p-1}{2})^2}{2n-p} \rfloor ^{p+q-1} + O(s^{p+q-2})$, which is the analog of $(5.6)$. The passage from $(5.19)$ to $(5.13)$ is the same as the passage from $(5.6)$ to $(5.11)$ given in the proof of Theorem 5.2. 

The verification that the lower order terms have been omitted contribute only $O(s^{p+q-2})$ is mostly routine. For example, the sum of $k^{q-1+m}n^l$ for $k$ in the range $(5.18)$ yields $\sum\limits_n n^l (a_n^{q+m} - a_{n+1}^{q+m}) = l\sum\limits_n n^{l-1}a_n^{q+m}$ plus lower order terms. Since $m+l \leq p$ the largest value comes from taking $l = p-m$, so for each $m \leq p-2$ we need to control $\sum\limits_n n^{p-m-1}a_n^{q+m}$, and $a_n \leq \frac{s}{2n-p}$ so we obtain $\sum\limits_{n = \frac{p+1}{2}}^{\sqrt{s}} s^{q+m} n^{p-q-2m-1}$. If $p - q - 2m - 1 \leq -2$ then the infinite series converges and we get the estimate $O(s^{q+m}) \leq O(s^{p+q-2})$. If $p-q - 2m -1 \geq 0$ then sum over $n$ is $O(s^{\frac{p-q-2m}{2}})$ so altogether we get $O(s^{\frac{p+q}{2}}) \leq O(s^{p+q-2})$  since $p + q \geq 5$. The remaining case $p - q - 2m = 0$ does not occur because $p - q$ is odd. 

The argument for $N^-(s)$ is essentially the same, just permuting the roles of $j$ and $k$.
\end{proof}
Figure 5.3 shows the graph of $R_{N^+}(s)/s^5$ for $p=3$, $q=4$. it does not appear that a limit exists as $s\rightarrow \infty$, although the oscillation is small. Averaging improves the behavior considerably, as it did for the case $p=1$,$q=4$ discovered earlier. Again we speculate that $(5.13)$ may be improved to

$(5.13')$ $N^+(s) = \frac{4}{(p+q-1)!(p-1)!(q-1)!}(1-2^{-(p+q-1)})\zeta(p+q-1)s^{p+q-1} + c(p,q)s^{p+q-2} + \tilde{R}_{N^+}(s)$ with $\tilde{R}_{N^+}(s) = O(s^{p+q-2-\delta(p,q)})$. This is illustrated in Figure 5.4 with numerically approximated values for $c(3,4)$ and $\delta(3,4) \approx 1.2$.
	\begin{figure}
	\hspace*{-.35in}
    \includegraphics[scale = .8]{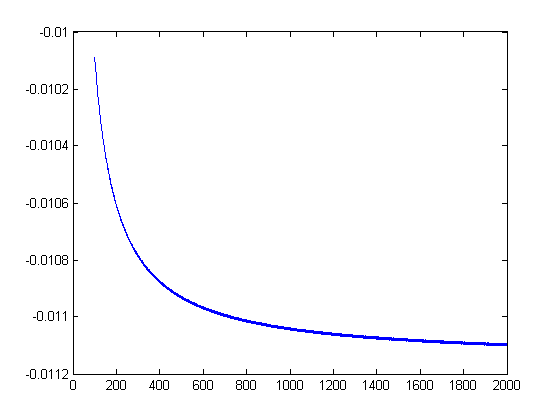}
   \caption{$R_{N^{+}}(s)/s^5$ from $(5.19)$ on $[100,2000]$ with $p = 3$ and $q = 4$}
\end{figure}
	\begin{figure}
	\hspace*{-.35in}
    \includegraphics[scale = .8]{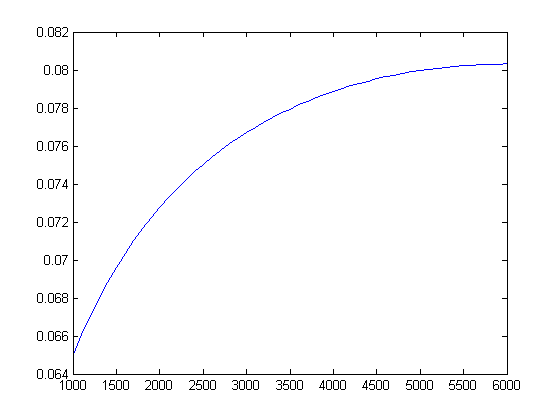}
   \caption{$(AR_{N^{+}}(s) + .00185845s^5)/s^{3.8}$ from $(5.19)$ on $[100,6000]$ with $p = 3$ and $q = 4$ and a sampling interval of 100}
\end{figure}
\section{Higher Order Equations}
Consider the operator $L_1 = (\frac{\partial}{\partial t})^4 + \bigtriangleup_x$ on $S^1 \times S^2$. It has the same eigenfunctions as $\Box$, but the eigenvalues are $j^4 - k(k+1)$ with the multiplicity $2k+1$. Here the 0-eigenspace consists of the constants, so it has multiplicity one. We define $N^+(s)$ and $N^-(s)$ as before, so

$(6.1)$ $N^+(s) = 2\sum\limits_k (2k+1)$ \#$\{ j > 0; 0 < j^4 - k(k+1) \leq s\}$ and 

$(6.2)$ $N^-(s) = \sum\limits_{k=0}^{\lfloor \sqrt{s + \frac{1}{4}} - \frac{1}{2} \rfloor} (2k+1) + 2\sum\limits_k(2k+1)$ \#$\{ j > 0; 0 < k(k+1) - j^4 \leq s\}$ with the first sum corresponding to $j = 0$.

\begin{lem}
For $L_1$ we have 

$(6.3)$ $N^+(s) = 2\sum\limits_{j = 1}^{\lfloor s^{1/4} \rfloor}j^4 + 2\sum\limits_{j= \lfloor s^{1/4} \rfloor + 1} ^{\lfloor s^{1/2} \rfloor} (j^4 - \lceil \sqrt{j^4 - s + \frac{1}{4}} - \frac{1}{2} \rceil^2)$
\end{lem}

\begin{proof}
We fix $j >0$ and find bounds for $k$. Note that $k(k+1) < j^4$ is equivalent to $k \leq j^2 - 1$, while $j^4 - k(k+1) \leq s$ is equivalent to $\sqrt{j^4 - s + \frac{1}{4}} - \frac{1}{2} \leq k$, and since $k$ is an integer this is the same as $\lceil \sqrt{j^4 - s + \frac{1}{4}} - \frac{1}{2} \rceil \leq k$. When $j \leq \lfloor s^{1/4} \rfloor$, the lower bound is just $k \geq 0$, so $2\sum\limits_{k=0}^{j^2 - 1}(2k+1) = 2j^4$, and this contributes the first sum in $(6.3)$. When $j \geq \lfloor s^{1/4} \rfloor + 1$ then we observe that the condition $\sqrt{j^4 - s + \frac{1}{4}} - \frac{1}{2} \leq j^2 - 1$ is equivalent to $j \leq \sqrt{s}$, so this gives the upper bound for $j$ in the second sum in $(6.3)$.
\end{proof}

\begin{thm}
For $L_1$ we have the asymptotic estimate 

$(6.4)$ $\frac{4}{3}s^{3/2} \leq N^{+}(s) \leq \frac{8}{3}s^{3/2}$ as $s \rightarrow \infty$.
\end{thm}

\begin{proof}
The first sum in $(6.3)$ is $O(s^{5/4})$. For the second sum we write $\lceil \sqrt{j^4 - s + \frac{1}{4}} - \frac{1}{2} \rceil = \sqrt{j^4 - s + \frac{1}{4}} + \eta_j - \frac{1}{2}$ for $0 \leq \eta_j < 1$ so $j^4 - \lceil \sqrt{j^4 - s + \frac{1}{4}} - \frac{1}{2} \rceil^{2} = s - \frac{1}{4} + (1-2\eta_j)\sqrt{j^4 - s + \frac{1}{4}} - (\eta_j - \frac{1}{2})^2,$ so the second sum is 

$(6.5)$ $2s^{3/2} + 2\sum\limits_{j = \lfloor s^{1/4} \rfloor + 1}^{\lfloor s^{1/2} \rfloor} (1- 2 \eta_j)\sqrt{j^4 - s + \frac{1}{4}} + O(s^{5/4})$. An upper bound for the sum in $(6.5)$ is $2\sum\limits_{j=0}^{\lfloor s^{1/2} \rfloor} j^2 = \frac{2}{3}s^{3/2}$, and a lower bound is $-\frac{2}{3}s^{3/2}$, so that yields $(6.4)$. 
\end{proof}

Figure $6.1$ shows the graph of $N^+(s)/s^{3/2}$. It seems unlikely that a limit exists as $s \rightarrow \infty$. Figure $6.2$ shows the average of the graph also divided by $s^{3/2}$. Now it appears likely that a limit exists, but the convergence is too slow to allow us to guess what the limit might be.
\begin{figure}
	\hspace*{-.35in}
    \includegraphics[scale = .8]{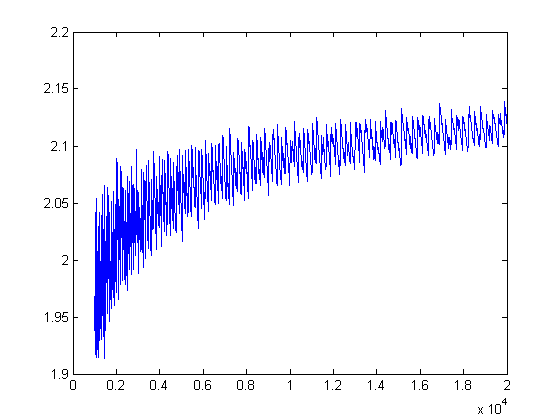}
   \caption{$N^{+}(s)/s^{3/2}$ from $(6.3)$ on $[1000,20000]$}
\end{figure}
\begin{figure}
	\hspace*{-.35in}
    \includegraphics[scale = .8]{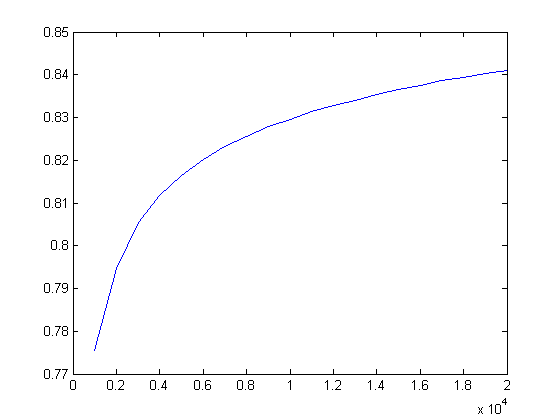}
    \caption{$A_{N^{+}}(s)/s^{3/2}$ from $(6.3)$ on $[1000,20000]$ sampling every 1000th integer}
\end{figure}

\begin{lem}
For $L_1$ we have 

$(6.6)$ $N^-(s) = \lfloor \sqrt{s + \frac{1}{4}} + \frac{1}{2} \rfloor ^2 + 2 \sum\limits_{j=1}^{\lfloor \sqrt{s} \rfloor} \left( \lfloor \sqrt{s + \frac{1}{4} + j^4} + \frac{1}{2} \rfloor ^2 - j^4 \right)$
\end{lem}

\begin{proof}
The first sum in $(6.2)$ contributes exactly $\lfloor \sqrt{s + \frac{1}{4}} + \frac{1}{2}\rfloor ^2$ to $N^-(s)$. For the second sum we fix $j$ and find the bounds on $k$. Note that $j^4 < k(k+1)$ is equivalent to $k \geq j^2$ while $k(k+1) \leq s + j^4$ is equivalent  to $k \leq \lfloor \sqrt{s + \frac{1}{4} + j^4} - \frac{1}{2}]$, while the condition $j^2 \leq \sqrt{s + \frac{1}{4} + j^4} - \frac{1}{2}$ is equivalent to $j \leq \lfloor \sqrt{s} \rfloor$, so the second sum in $(6.2)$ is equal to the second sum in $(6.6)$.
\end{proof}

\begin{thm}
For $L_1$ we have the asymptotic estimate 

$(6.7)$ $\frac{4}{3}s^{3/2} \leq N^-(s) \leq \frac{8}{3}s^{3/2}$ as $s \rightarrow \infty$.
\end{thm}

\begin{proof}
The first term in $(6.6)$ is $O(s)$. To estimate the sum in $(6.6)$ we write $\lfloor \sqrt{s + \frac{1}{4} + j^4} + \frac{1}{2} \rfloor = \sqrt{s + \frac{1}{4} + j^4} + \frac{1}{2} - \eta_j$ for $0 < \eta_j \leq 1$. The sum is exactly 

$(6.8)$ $2\sum\limits_{j=1}^{\lfloor \sqrt{s} \rfloor} \left( s + (1-2\eta_j)\sqrt{s + \frac{1}{4} + j^4} + \frac{1}{4} + (\frac{1}{2} - \eta_j)^2\right) = 2s^{3/2} + 2\sum\limits_{j=1}^{\lfloor \sqrt{s} \rfloor} (1 - 2\eta_j) \sqrt{s + \frac{1}{4} + j^4} + O(s^{1/2})$

We use the estimate $\sqrt{s + \frac{1}{4} + j^4} \leq j^2 + \sqrt{s}$ in $(6.8)$ to obtain $(6.7)$.
\end{proof}
Figure $6.3$ shows the graph of $N^-(s)/s^{3/2}$. Again it appears unlikely that a limit exists. On the other hand, the average divided by $s^{3/2}$ appears to converge rapidly enough that we can estimate the limit to be $.9008175$. Figure 6.4 shows the graph of the remainder divided by $s^{3/2}$.
\begin{figure}
	\hspace*{-.35in}
    \includegraphics[scale = .8]{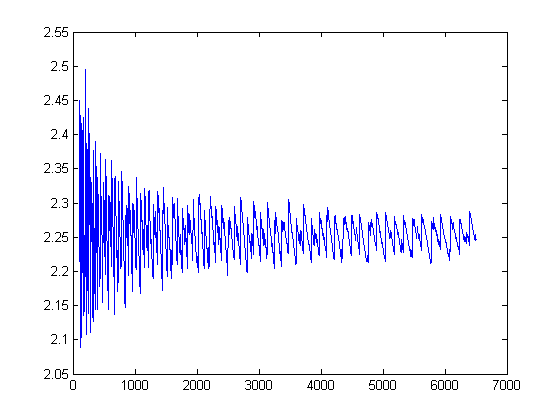}
   \caption{$N^{-}(s)/s^{3/2}$ from $(6.6)$ on $[100,6500]$}
\end{figure}
\begin{figure}
	\hspace*{-.35in}
    \includegraphics[scale = .8]{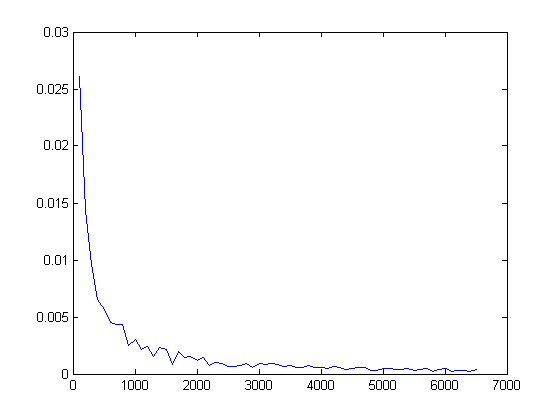}
   \caption{$(A_{N^{-}}(s) -.9008175 *s^{3/2}) /s^{3/2}$ from $(6.6)$ on $[100,6500]$ with a sampling interval of 100}
\end{figure}

Finally we consider the operator $L_2 = -\frac{\partial^2}{\partial t^2} - \bigtriangleup ^2_x$. Again it has the same eigenfunctions and multiplicities, with the eigenvalues now equal to $j^2- (k(k+1))^2$. We notice that the 0-eigenspace is infinite dimensional, corresponding to $|j| = k(k+1)$. Just as in the first example in section 3, if we ignore this we can still consider $N^+(s)$ and $N^-(s)$,

$(6.9)$ $N^+(s) = \sum(2k+1)$ over all solutions to $0 < j^2 - (k(k+1))^2 \leq s$,

$(6.10)$ $N^-(s) = \sum(2k+1)$ over all solutions to $0 < (k(k+1))^2 - j^2 \leq s$.

\begin{lem}
For $L_2$ we have 

$(6.11)$ $N^+(s) = 2\sum\limits_{n=1}^{\lfloor \sqrt{s} \rfloor} \lfloor \frac{1 + \sqrt{1 + \frac{2}{n}(s-n^2)}}{2} \rfloor ^2$.
\end{lem}
\begin{proof}
The condition $(k(k+1))^2 < j^2$ is equivalent to $k(k+1) + 1 \leq |j|$, and the condition $j^2 \leq s + (k(k+1))^2$ is equivalent to $|j| \leq \lfloor \sqrt{s + (k(k+1))^2} \rfloor$. In order to have at least one solution to both inequalities, namely $j = k(k+1) + 1$, we must have $k(k+1) + 1 \leq \sqrt{s + (k(k+1))^2}$, which is equivalent to $k \leq \lfloor \frac{\sqrt{2s-1} - 1}{2} \rfloor$. So $(6.9)$ becomes 

$(6.12)$ $N^+(s) = 2\sum\limits_{k=0}^{\lfloor \frac{\sqrt{2s-1} -1}{2} \rfloor} (2k+1)(\lfloor \sqrt{s+(k(k+1))^2} - k(k+1) \rfloor)$. 

Now define $n$ by $\sqrt{s + (k(k+1))^2} = n + k(k+1) + \delta$ for $0 \leq \delta < 1$. Solving for $k$ we obtain $k = \frac{-1 + \sqrt{1 + \frac{2}{n+\delta}(s - (n + \delta)^2)}}{2}$, which yields the bounds $\lfloor \frac{1 + \sqrt{1 + \frac{2}{n+1}(s - (2n+1)^2)}}{2} \rfloor $

$\leq k \leq \lfloor \frac{-1 + \sqrt{1 + \frac{2}{n}(s-n^2)}}{2} \rfloor$, so $(6.12)$ becomes $N^+(s) = 2\sum\limits_{n=1}^{\lfloor \sqrt{s} \rfloor} n(a_n^2 - a^2_{n+1})$ 

$= 2\sum\limits_{n=1}^{\lfloor \sqrt{s} \rfloor} a_n^2$ for $a_n = \lfloor \frac{1 + \sqrt{1 + \frac{2}{n}(s-n^2)}}{2} \rfloor$, and this yields $(6.11)$.
\end{proof}
\begin{thm}
For $L_2$ we have the asymptotic formula 

$(6.13)$ $N^+(s) = \frac{1}{2}s\log{s}+(\gamma - \frac{1}{2})s + R_{N^+}(s)$ with the remainder estimate

$(6.14)$ $R_{N^+}(s) = O(s^{3/4})$.
\end{thm}
\begin{proof}
Write $\lfloor \frac{1 + \sqrt{1 + \frac{2}{n}(s-n^2)}}{2} \rfloor = \frac{1 + \sqrt{1 + \frac{2}{n}(s-n^2)}}{2} - \eta_n$ with $0 \leq \eta_n < 1$. Then $(6.11)$ becomes 

$N^+(s) = \sum\limits_{n=1}^{\lfloor \sqrt{s} \rfloor} \left( \frac{1 + \frac{2}{n}(s-n^2)}{2} + (1-2\eta_n)\sqrt{1 + \frac{2}{n}(s-n^2)} + 2(\eta_n - \frac{1}{2})^2 \right)$ 

$= \sum\limits_{n=1}^{\lfloor \sqrt{s} \rfloor}(\frac{s}{n} - n) + \sum\limits_{n=1}^{\lfloor \sqrt{s} \rfloor}(1-2\eta_n)\sqrt{1 + \frac{2}{n}(s-n^2)} + O(\sqrt{s})$ and $\sum\limits_{n=1}^{\lfloor \sqrt{s} \rfloor}(\frac{s}{n} -n)$ 

$= s(\log{\sqrt{s}} + \gamma - \frac{1}{2}) + O(\sqrt{s})$, giving the main terms in $(6.13)$. On the other hand, $|(1- 2\eta_n)\sqrt{1 + \frac{2}{n}(s-n^2)}| \leq \sqrt{\frac{2s}{n}}$, so the remaining sum is dominated by $\sum\limits_{n=1}^{\lfloor \sqrt{s} \rfloor} \sqrt{\frac{2s}{n}} = O(s^{3/4})$, which yields the error estimate $(6.14)$.

Remark: Because $1 - 2\eta_n$ has "mean value" equal to zero, we expect the actual error to be smaller than the estimate $(6.14)$.
\end{proof}
Figure $6.5$ shows the graph of $R_{N^+(s)}/s^{3/4}$. Figure $6.6$ shows its average also divided by $s^{3/4}$. It appears possible that the limit exists for the average, but the evidence is ambiguous.
\begin{figure}
	\hspace*{-.35in}
    \includegraphics[scale = .8]{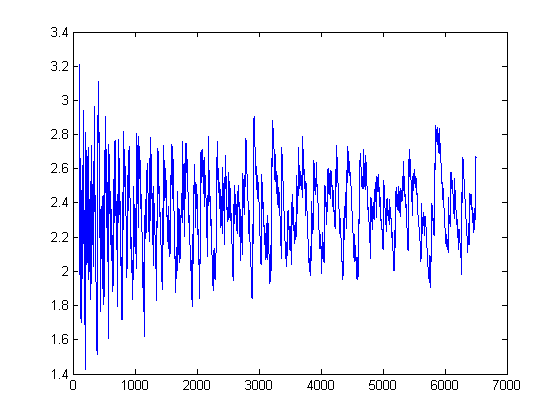}
   \caption{$R_{N^{+}}(s)/s^{3/4}$ from $(6.13)$ on $[100,6500]$}
\end{figure}
\begin{figure}
	\hspace*{-.35in}
    \includegraphics[scale = .8]{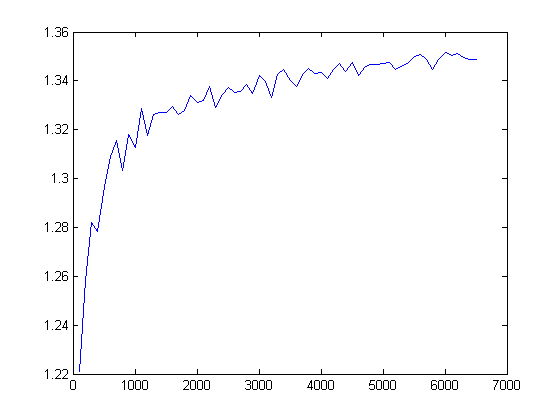}
   \caption{$AR_{N^{+}}(s)/s^{3/4}$ from $(6.13)$ on $[100,6500]$ with a sampling interval of 100}
\end{figure}
\begin{lem}
For $L_2$ we have 

$(6.15)$ $N^-(s) = \sum\limits_{k = 1}^{\lfloor \frac{\sqrt{1 + \sqrt{s}} - 1}{2} \rfloor} (2k+1)(2k^2 + 2k - 1) + 2\sum\limits_{n=1}^{\lfloor \sqrt{s} \rfloor - 1} \lfloor \frac{1 + \sqrt{1 + \frac{2}{n}(s + n^2)}}{2} \rfloor^2$
\end{lem}
\begin{proof}
The condition $j^2 < (k(k+1))^2$ is equivalent to $k \geq 1$ and $|j| \leq k(k+1) - 1$. For the condition $(k(k+1))^2 \leq s + j^2$ we consider the cases:
\begin{enumerate}[I.]
\item
If $4 \leq (k(k+1))^2 \leq s$ then the condition is always valid, so the total number of $j$ values is $2k^2 + 2k -1$, and this contributes the first sum in $(6.15)$.
\item
If $(k(k+1))^2 > s$ then the condition is $|j| \geq \lceil \sqrt{(k(k+1))^2 - s} \rceil$, so the total number of $j$ values is $2(k(k+1) - \lceil \sqrt{(k(k+1))^2 - s} \rceil)$. We define $n$ by $\sqrt{(k(k+1))^2 - s} = k(k+1) - n - \delta$ with $0 \leq \delta < 1$, so the contribution to $N^-(s)$ from these terms is $2\sum(2k+1)n$ over the appropriate values of $k$ and $n$. If we solve for $k$ in terms of $n$ we obtain $k = \frac{-1 + \sqrt{1 + \frac{2}{n+ \delta}(s + (n+\delta)^2)}}{2}$, so the range of $k$ values is $\lfloor \frac{1 + \sqrt{1 + \frac{2}{n+1}(s + (n+1)^2)}}{2} \rfloor \leq k \leq \lfloor \frac{-1 + \sqrt{1 + \frac{2}{n}(s+n^2)}}{2} \rfloor$, and the sum over $k$ is $a_n^2 - a_{n+1}^2$ for $a_n = \lfloor \frac{1 + \sqrt{1 + \frac{2}{n}(s+n^2)}}{2} \rfloor$, which yields the sum $2\sum\limits_n a_n^2$. To obtain the bounds on $n$ we note that $n$ decreases as $k$ increases so the maximal value of $n$ occurs when $k(k+1)$ first exceeds $\sqrt{s}$, and $k(k+1) = \sqrt{s}$ corresponds to $n = \lfloor \sqrt{s} \rfloor$.
\end{enumerate}
\end{proof}
\begin{thm}
For $L_2$ we have the asymptotic formula 

$(6.16)$ $N^-(s) = \frac{1}{2}s\log{s} + (\gamma + \frac{3}{2})s + R_{N^-}(s)$ with remainder estimate

$(6.17)$ $R_{N^-}(s) = O(s^{3/4})$.
\end{thm}
\begin{proof}
The first sum in $(6.15)$ is $s + O(s^{3/4})$. As in the proof of Theorem 6.6, we see that the second sum in $(6.15)$ is $\sum\limits_{n=1}^{\lfloor \sqrt{s} \rfloor} (\frac{s}{n} + n) + O(s^{3/4}) = s(\log\sqrt{s} + \gamma + \frac{1}{2}) + O(s^{3/4})$. Add.
\end{proof}
Figure $6.7$ shows the graph of $R_{N^+(s)}/s^{3/4}$. Figure $6.8$ shows its average, also divided by $s^{3/4}$. It is not apparent from Figure 6.7 that the ratio is bounded from below. On the other hand, from Figure 6.8 we speculate that the ratio $AR_{N^-}(s)/s^{3/4}$ is decreasing, so a limit would exist, although we are not able to estimate what it would be.
\begin{figure}
	\hspace*{-.35in}
    \includegraphics[scale = .8]{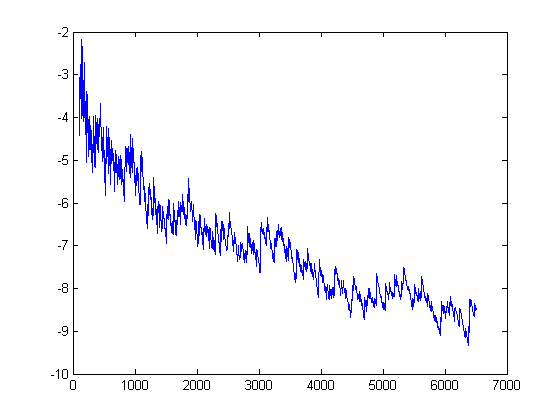}
   \caption{$R_{N^{-}}(s)/s^{3/4}$ from $(6.16)$ on $[100,6500]$}
\end{figure}
\begin{figure}
	\hspace*{-.35in}
    \includegraphics[scale = .8]{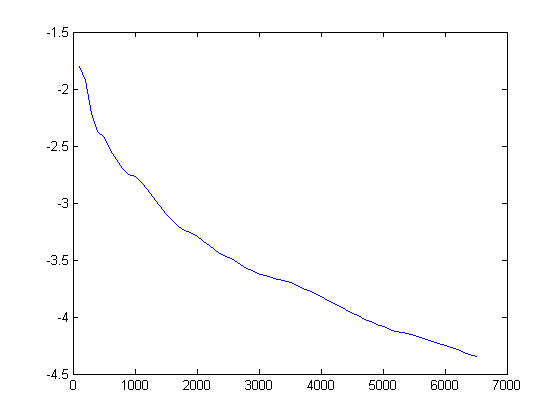}
   \caption{$AR_{N^{-}}(s)/s^{3/4}$ from $(6.16)$ on $[100,6500]$ with a sampling interval of 100}
\end{figure}
 
\end{document}